\documentclass{amsart}
\usepackage{amssymb, amsfonts, lacromay}
\usepackage{mathrsfs,comment}
\usepackage[usenames,dvipsnames]{color}
\usepackage[normalem]{ulem}
\usepackage{url}
\usepackage{tikz-cd}
\usepackage[all,arc,2cell]{xy}
\UseAllTwocells
\usepackage{enumerate}

\usepackage{todonotes}
\usepackage{hyperref}  
\hypersetup{%
  bookmarksnumbered=true,%
  bookmarks=true,%
  colorlinks=true,
  linkcolor=blue,%
  citecolor=blue,%
  filecolor=blue,%
  menucolor=blue,%
  pagecolor=blue,%
  urlcolor=blue,%
  pdfnewwindow=true,%
  pdfstartview=FitBH}   
  
\usepackage[inline]{showlabels}

\def\makeautorefname#1#2{\expandafter\def\csname#1autorefname\endcsname{#2}}
\def\equationautorefname~#1\null{(#1)\null}
\makeautorefname{footnote}{footnote}%
\makeautorefname{item}{item}%
\makeautorefname{figure}{Figure}%
\makeautorefname{table}{Table}%
\makeautorefname{part}{Part}%
\makeautorefname{appendix}{Appendix}%
\makeautorefname{chapter}{Chapter}%
\makeautorefname{section}{Section}%
\makeautorefname{subsection}{Section}%
\makeautorefname{subsubsection}{Section}%
\makeautorefname{theorem}{Theorem}%
\makeautorefname{thm}{Theorem}%
\makeautorefname{sta}{Statement}%
\makeautorefname{cor}{Corollary}%
\makeautorefname{lem}{Lemma}%
\makeautorefname{prop}{Proposition}%
\makeautorefname{pro}{Property}
\makeautorefname{conj}{Conjecture}%
\makeautorefname{defn}{Definition}%
\makeautorefname{notn}{Notation}
\makeautorefname{notns}{Notations}
\makeautorefname{rem}{Remark}%
\makeautorefname{rmk}{Remark}%
\makeautorefname{rems}{Remarks}%
\makeautorefname{quest}{Question}%
\makeautorefname{exmp}{Example}%
\makeautorefname{ax}{Axiom}%
\makeautorefname{claim}{Claim}%
\makeautorefname{assn}{Assumption}%
\makeautorefname{asses}{Assumptions}%
\makeautorefname{con}{Construction}%
\makeautorefname{prob}{Problem}%
\makeautorefname{warn}{Warning}%
\makeautorefname{obs}{Observation}%
\makeautorefname{conv}{Convention}%
\makeautorefname{cont}{Context}%

\newtheorem{thm}{Theorem}[section]

\newtheorem{cor}{Corollary}[section]
\newtheorem{prop}{Proposition}[section]
\newtheorem{lem}{Lemma}[section]

\theoremstyle{definition}
\newtheorem{defn}{Definition}[section]

\newtheorem{con}{Construction}[section]
\newtheorem{exmp}{Example}[section]
\newtheorem{notn}{Notation}[section]
\newtheorem{notns}{Notations}[section]

\newtheorem{quest}{Question}[section]
\newtheorem{rem}{Remark}[section]

\newtheorem{sch}{Scholium}[section]

\newtheorem{conv}{Convention}[section]
\newtheorem{cont}{Context}[section]

\newcounter{assn}[section]
\renewcommand{\theassn}{\Alph{assn}}

\makeatletter
\let\c@cont=\c@thm
\let\c@conv=\c@thm
\let\c@obs=\c@thm
\let\c@sta=\c@thm
\let\c@cor=\c@thm
\let\c@prop=\c@thm
\let\c@lem=\c@thm
\let\c@prob=\c@thm
\let\c@con=\c@thm
\let\c@conj=\c@thm
\let\c@defn=\c@thm
\let\c@notn=\c@thm
\let\c@notns=\c@thm
\let\c@exmp=\c@thm
\let\c@ax=\c@thm
\let\c@pro=\c@thm
\let\c@quest=\c@thm
\let\c@ass=\c@thm
\let\c@warn=\c@thm
\let\c@rem=\c@thm
\let\c@sch=\c@thm
\let\c@equation\c@thm
\numberwithin{equation}{section}
\makeatother

\definecolor{orange}{rgb}{1,0.5,0}
\newcommand{\pedit}[1]{{\color{red}{#1}}}

\newcommand{\Cmon}{\mathbb{C}}
\newcommand{\Cp}{\mathbb{C}^{\mathrm{pre}}}

\newcommand{\Cpre}{\hat{\mathbb{C}}^{\mathrm{pre}}}
\newcommand{\CpreGH}{\hat{\mathbb{C}}^{\mathrm{pre}}_{G/H}}
\newcommand{\sCpre}{\mathscr{C}^{\mathrm{pre}}}
\newcommand{\sCprest}{\mathscr{C}^{\mathrm{pre}}_{\star}}
\newcommand{\sCpGH}{\mathscr{C}^{\mathrm{pre}}_{G/H}}
\newcommand{\sCpreGH}{\mathscr{C}^{\mathrm{pre}}_{G/H}}

\newcommand{\NewC}{\mathscr{C}^{fin}_G}

\newcommand{\bj}{\mathbf{j}}
\newcommand{\bk}{\mathbf{k}}
\newcommand{\bm}{\mathbf{m}}
\newcommand{\bn}{\mathbf{n}}
\newcommand{\bp}{\mathbf{p}}

\newcommand{\OG}{G\mathscr{O}}
\newcommand{\LG}{G\LA}
\newcommand{\LH}{H\LA}

\newcommand{\LRG}{G\tilde{\LA}}
\newcommand{\LRH}{H\tilde{\LA}}
\newcommand{\LRK}{K\tilde{\LA}}
\newcommand{\FRG}{G\tilde{\mathscr{F}}}
\newcommand{\OGop}{G\mathscr{O}^{op}[\sT]}
\newcommand{\Cprea}{\bC^{pre}\big[\OGop\big]}
\newcommand{\Cpa}{\bC^{pre}\big[\OGop\big]}

\newcommand{\generalE}{\mathcal{E}} 
\newcommand{\fixpt}{\mathbb{R}} 
\newcommand{\underlying}{\mathbb{L}} 
\newcommand{\originalE}{\mathcal{E}_{orig}} 
\newcommand{\ourE}{\mathcal{E}} 
\newcommand{\ourEonAlg}{{\mathcal{\sE}_\sC}} 

\title{Orbital presheaves in equivariant infinite loop space theory}

\author{Hana Jia Kong}
\address{School of Mathematical Sciences, Zhejiang University, Hangzhou, China}
\email{hana.jia.kong@gmail.com}

\author{J. Peter May}
\address{Department of Mathematics, The University of Chicago, Chicago, IL 60637}
\email{may@math.uchicago.edu}

\author{Foling Zou}
\address{Institute of Mathematics, Chinese Academy of Sciences, Beijing, China}
\email{zoufoling@amss.ac.cn}


\begin{document}

\begin{abstract}
Let $G$ be a finite group. Using a new kind of operad, we axiomatize and explore an infinite loop space machine that constructs (genuine) $G$-spectra from suitably structured functors on orbital presheaves (contravariant functors from the orbit category of $G$ to  based spaces). The theory leads unexpectedly to a new operadic description of Mackey functors and hence to a definition of ``topological Mackey functors" and a construction of their associated $G$-spectra.  It also leads to Picard $G$-spectra, Azumaya ring $G$-spectra, and Brauer $G$-spectra. These constructions raise many unanswered questions.
\end{abstract} 

\maketitle

\tableofcontents

\section*{Introduction}

Constructing spectra out of structured space level data has long been important in algebraic topology. Methods of construction are called infinite loop space
machines.  There are two main flavors: the operadic machine and the Segal machine.  Introducing a common generalization, they were long ago proven to be equivalent \cite{MT}.  Generalizing such machines equivariantly began in the 1980's and has had a checkered career.  A modern treatment of the equivariant operadic machine is given in \cite{GM3}.  A modern treatment of the equivariant Segal machine and the common generalization, together with a direct point-set level proof of the equivalence of these equivariant machines is given in \cite{MMO}.   The input in these sources is structured $G$-spaces or suitably structured functors to $G$-spaces. The structure is given by actions by monads associated to operads or to categories of operators.  We take $G$ to be finite throughout.

However, the input that one sees naturally is often given by contravariant functors from the orbit category $\OG$ to spaces, not $G$-spaces.    Here the objects and morphisms of $\OG$ are the orbits $G/H$ and the $G$-maps between them.   We call such functors ``orbital presheaves''.   The first infinite loop space machine taking such presheaves as input is in a quite brief 1991 paper of Costenoble and Waner \cite{CW}.   It makes for hard reading for several reasons.  One is that it was written well before its time: equivariant theory was still in its infancy. Another is that it is written combinatorially, with relatively little conceptual foundation.   It relied on a never published preprint of Costenoble, Hauschild, May, and Waner for the operadic $G$-space input machine, on which the operadic presheaf machine relies.   

The work of Costenoble and Waner starts with a (genuine) $E_{\infty}$ operad $\sC$ of $G$-spaces\footnote{We take operads to be reduced, so that $\sC(0)$ is a point.}, as defined\footnote{The original definition is in the generality of compact Lie groups, the motivation already back then being to define and study $E_{\infty}$ ring $G$-spectra in that generality.} in \cite[Chapter VII]{LMS}, and constructs from it a monad $\Cp$ of orbital presheaves that is nicely related to the monad $\bC$ of $G$-spaces associated to $\sC$.  That construction is not easy, but it is beautiful.  We shall rework it conceptually using more modern categorical understanding and a new kind of equivariant operad.

For reasons explained in detail in \cite{KMZ1}, it is most natural here to take the category $G\sS$ of $G$-spectra to refer to Lewis--May $G$-spectra \cite{LMS}. These consist of based $G$-spaces  $E_V$ for finite dimensional sub inner product spaces $V$ of  a complete universe $U$, meaning an inner product space that is the sum of countably many copies of each irreducible representation of $G$, together with $G$-homeomorphisms $E_V \rtarr \OM^{W-V} E_{W}$ for $V\subset W$.   Viewed as an $\infty$-category, it is equivalent to any good modern category of $G$-spectra, but its transparently elementary concrete adjunction $(\SI^{\infty},\OM^{\infty})$, where $\OM^{\infty}E = E_0$, is unavailable with other choices. 

We assume that the operad $\sC$ has a natural action $\vartheta  \colon \bC E_0 \rtarr E_0$ on equivariant infinite loop spaces, where $\bC$ is the monad associated to $\sC$.  The most sensible choice is the equivariant Steiner operad \cite{Steiner, GM3} or its product with any other $E_{\infty}$ $G$-operad.  This places us in an example of the general context  of  \cite[Assumption A]{KMZ1}, and more complete statements of the following recognition principle can be found there or in \cite[\S2.3]{GM3}.

Let $G\sT$ denote the category of based $G$-spaces and based $G$-maps, and let $\bC[G\sT]$ denote the category of $\bC$-algebras in $G\sT$.\footnote{As in \cite{KMZ1}, we ignore the distinction between based $G$-spaces and nondegenerately based $G$-spaces since it is easily circumvented without loss of generality and only obscures the exposition.} 

\begin{thm}\label{machine}  Let $G$ be a finite group. For an $E_{\infty}$ $G$-operad $\sC$ which acts naturally on equivariant infinite loop spaces, there is a functor
 $\mathbb{E}: \bC[G\sT] \rtarr G\sS$ such that, for $X \in  \bC[G\sT]$, there is a natural group completion 
$$\ze\colon X \rtarr \Omega^{\infty}\mathbb{E}X.$$  Therefore $\ze$ is a weak equivalence if $X$ is grouplike, meaning that 
$\pi_0(X^H)$ is a group for  each subgroup $H$ of $G$.
\end{thm}

\begin{rem} While $\ze$ is not a map of $\bC$-algebras, it is the composite of a map of $\bC$-algebras $\overline X \rtarr  \Omega^{\infty}\mathbb{E}X$ and the homotopy inverse of a map of $\bC$-algebras $\overline{X}\rtarr X$ which is a homotopy equivalence of underlying $G$-spaces.  Here $\overline{X} = B(\bC,\bC,X)$ is a homotopically well-behaved approximation of $X$. 
\end{rem}

We say that an  orbital presheaf $\sX$ is grouplike if each $\pi_0\sX(G/H)$ is a group and we say that a map $\sX\rtarr \sY$ from $\sX$ to a grouplike orbital presheaf $\sY$ is a group completion if each $\sX(G/H) \rtarr \sY(G/H)$ is a group completion in the usual\footnote{\cite{KMZ1} gives a new conceptual definition, with and without group actions.} nonequivariant sense: the induced map 
$$ H_*(\sX(G/H))[\pi_0(\sX(G/H))^{-1}] \rtarr H_*(\sY(G/H))$$ 
is an isomorphism for all subgroups $H$ of $G$.  

We obtain the orbital presheaf recognition principle as a corollary.   Write $G\sO^{op}[\sT]$ for the category of orbital presheaves and write  $\Cprea$ for the category of algebras over the monad $\Cp$ in $G\sO^{op}[\sT]$.  That is the domain category of our orbital presheaf infinite loop space machine.  Recall that the fixed point functor $\bR\colon G\sT \rtarr G\sO^{op}[\sT]$ sends a based $G$-space $X$ to the orbital presheaf that sends $G/H$ to $X^H$.  By \autoref{Acom}, $\bR$ restricts on $\bC$-algebras in $G\sT$ to a functor 
$$\bR \colon \bC[G\sT] \rtarr \Cprea.$$  
We shall construct a functor  
\begin{equation}\label{ElmC}
\sE_{\bC}\colon \Cprea \rtarr \bC[G\sT].
\end{equation}
It is a new version of the Elmendorf construction.  It sends a $\Cp$-algebra $\sX$ to a $\bC$-algebra such that $\bR\sE_{\bC}\sX$ is weakly equivalent to $\sX$.  Composing the machine $E$ with  $\sE_{\bC}$ gives the orbital presheaf infinite loop space machine.
 
\begin{thm}
  \label{thm:pre-machine}
For an $E_{\infty}$ $G$-operad $\sC$ which acts naturally on equivariant infinite loop spaces, there is a  composite functor  
$$\mathbb{E}^{pre} = \mathbb{E} \circ \ourEonAlg \colon \Cprea \longrightarrow G\mathscr{S}$$
  such that, for $\sX \in \Cprea$, there is 
\begin{enumerate}[(i)]
\item a levelwise natural weak equivalence between $\bR \ourEonAlg \sX$ and $\sX$;
\item  a natural group completion 
$$\zeta^{pre}=\bR \zeta\colon \bR \ourEonAlg \sX \rtarr \bR \Omega^{\infty}\mathbb{E}\ourEonAlg \sX =   \bR
  \Omega^{\infty}\mathbb{E}^{pre} \sX.$$
Therefore $\zeta^{pre}$ is a levelwise weak equivalence when $\sX$ is grouplike.
\end{enumerate}
\end{thm}

As explained in \cite{KMZ1}, \autoref{machine} implies that $\bE$ induces an equivalence with inverse $\OM^{\infty}$ from the homotopy category of grouplike $\bC$-algebras to the homotopy category of connective spectra; \autoref{thm:pre-machine} implies the following analog. The question of how to extend equivalences such as this to equivalences of $\infty$-categories is discussed in \cite{KMZ1}.

\begin{thm}\label{equiv} The functor $\mathbb{E}^{pre}$ induces an equivalence with inverse $\bR\OM^{\infty}$ between the homotopy category of grouplike $\Cp$-algebras and the homotopy category of connective $G$-spectra.
\end{thm} 

Before we say how the paper is organized, we indicate the main examples, which are treated  in Sections 6, 7, and 8.  Section 6 concerns the simplest relevant operad, namely the commutativity operad $\sN$. Each $\sN(n)$ is a point, and $G$ acts trivially.   It is a surprise that the monad $\bN^{pre}$ on orbital presheaves has the following remarkable property (\autoref{cor:MackeyN}). 

\begin{thm}\label{Mackey}  The category of Mackey functors is isomorphic to the category of discrete grouplike $\bN^{pre}$-algebras.
\end{thm}

Discrete means taking values in sets regarded as discrete spaces. Although $\sN$ is not an $E_{\infty}$ operad, we can map to it from any $E_{\infty}$ operad.  Then \autoref{thm:pre-machine} gives a functorial construction of Eilenberg-Mac\,Lane $G$-spectra (\autoref{HM1}).

\begin{thm} The functor $\bE^{pre}$ restricts on Mackey functors $M$ to give a functorial construction of Eilenberg-Mac\,Lane $G$-spectra $HM$.
\end{thm}

Just as Mackey functors are the equivariant analog of abelian groups, deleting the discreteness condition in \autoref{Mackey} gives the equivariant analog of topological abelian groups, which deserve to be called topological Mackey functors.  A construction of their associated $G$-spectra works in exactly the same way (\autoref{HMTop}).

In \autoref{units}, we construct the unit $G$-spectrum of an $E_{\infty}$ ring $G$-spectrum $R$.  Nonequivariantly, these were first defined in \cite[Section IV.3]{MQR}, where they were used to understand obstructions in orientation theory.  They play an important role in much recent work, for example in \cite{ABGHR1, ABGHR2}.   We apply \autoref{thm:pre-machine} to give the equivariant generalization in \autoref{unit}.

\begin{thm}  For an $E_{\infty}$ $G$-spectrum $R$, there is a functorially constructed general linear  $G$-spectrum $gl_1(R)$ such that $\bR\OM^{\infty}  gl_1(R)$ is equivalent to $\ul{GL}_1(R)$.  That is, $(\OM^{\infty}gl_{1}(R))^H$ is equivalent to $GL_1(R^H)$, compatibly as $H$ varies.
\end{thm}
By the counterexample in \autoref{rem:not_equivalence},  this cannot be obtained directly from \autoref{machine}.
\begin{rem} This result with a version of our direct operadic proof has long been known, although never published.   Rekha Santhanam \cite{Sant} published a proof that translated our operadic input  to 
$\Gamma$-spaces and proceeded from there.
\end{rem}

In \autoref{PICBR}, we review the construction of equivariant $K$-theory spectra and construct the equivariant Picard and Brauer $G$-spectra of an $E_{\infty}$ ring $G$-spectrum $R$.  This can be viewed as a step in the spectrification of classical algebra.   Here the constructions are complete but their understanding is not.  This section is focused on the construction of $G$-spectra from permutative categories.  Just as there are classical (alias naive) and genuine $G$-spectra, there are also classical (alias naive) and genuine permutative $G$-categories. There is a functor that sends a classical permutative $G$-category $\sA$ to a genuine permutative $G$-category $\sA_G$.  (We generally omit the word genuine both on the category and the spectrum level.)    Our machine sends a classical permutative $G$-category 
$\sA$ to a classical $G$-spectrum $\bK \sA$ and sends a (genuine) permutative $G$-category $\sB$ to a (genuine) $G$-spectrum $\bK_G\sB$.  We abbreviate $\bK_G \sA_G$ to $\bK_G\sA$.  

We have a natural comparison $\io\colon \bK\sA \rtarr i^*\bK_G\sA$, where $i^*$ sends a genuine $G$-spectrum to its underlying classical $G$-spectrum.  It is an old observation \cite{GM3, Mer} that $\io$ is a weak equivalence in the context of the $K$-theory of Galois extensions, that being a direct consequence of Serre's version of Hilbert's Theorem 90.  In general, $B\io$ may or may not be a weak equivalence.\footnote{A recent preprint defines a categorical notion of weak equivalence that is so weak as to obliterate the question.  See Scholium \ref{Lenz}.} As implicitly observed by Merling \cite{Mer},  determining what $\io$ does can be viewed as a special case of Thomason's general homotopy limit problem \cite{Thom}.  We set up the context in \autoref{Perm} and discuss the homotopy limit problem in \autoref{lim}.  It determines whether or not the genuine $G$-spectra we construct have the fixed point spectra we would like them to have.  It remains open for the $K$-theory of $E_{\infty}$ ring $G$-spectra. See \autoref{Spec90}. 

We specialize to discuss Picard $G$-spectra in \autoref{PicR}, where the nonequivariant construction is  reviewed, and \autoref{GPicR}.  In parallel, we discuss Brauer $G$-spectra in \autoref{BrR}, where the nonequivariant construction is reviewed, and \autoref{GBrR}.  The parallel definitions of classical Picard and Brauer $G$-spectra of an $E_{\infty}$ ring spectrum are given in Propositions \ref{picprop} and \ref{brprop}, together with the verification that they have the nonequivariant  fixed point spectra that we want.  The parallel definitions of genuine Picard and Brauer $G$-spectra are given in Definitions \ref{PicG} and \ref{BrG}.  In neither case do we attempt to solve the homotopy limit problem. See Questions \ref{Pic90} and \ref{Br90}.  However, in Remarks \ref{Comp1} and \ref{Comp2}, we indicate how most of the problem reduces to the analogous open problem for  $K$-theory.  There is work to be done.  

This paper is a spin-off of the larger project \cite{KMZ1}.  There we give a general axiomatic framework for all of monadic iterated and infinite loop space theory, classical, equivariant, and multiplicative, under certain precisely specified categorical ``Assumptions".  Theorems \ref{machine}, \ref{thm:pre-machine}, and \ref{equiv} fit into that framework.   The specificity and importance of the orbital presheaf context suggested the value of this equivariant application.  Its key constructions are independent of the general framework and involve ideas of independent interest.  This paper is also a conceptual reworking and simplification of Costenoble--Waner \cite{CW}. 

As already said, for a G-space X, passage to fixed points gives its orbital presheaf 
$\bR X$.  We are interested in orbital presheaves $\sX$ that are {\em{not}} of that form.
We need the categorical preliminaries of Section \ref  {GSPACE} to work with such presheaves.  Our starting point is the equivalences $H\LA \iso G\LA_{G/H}$ of \autoref{prop:n-delta}, reexpressed with orderings in \autoref{Orderings}. Here $H\LA$ is the category of finite $H$-sets, defined independently of $G$, and $G\LA_{G/H}$ is the category of finite $G$-sets over $G/H$.  It is often significantly easier to work with $H\LA$, but to put things together as $H$ varies the latter is essential.   Use of orderings is also essential since equivariance is dictated homotopically by the relationship between $G$ and the symmetric groups.  

The theory  in \autoref{GSPACE} is an application of the classical  Grothendieck context in category theory.   To get to the main points more quickly, we relegate categorical background to the appendix, \autoref{CAT}, referring to it earlier where needed.  It seems remarkable to us that this categorical framework is so well adapted to our equivariant story.

In \autoref{CW1}, we focus on construction of an endofunctor $\Cpre$ on the category of orbital presheaves that specializes to $\Cp$, with its monad structure ignored.  There are several key choices to be made, and their interrelationships are not obvious.  Understanding them is central to the theory.  The story is outlined in \autoref{CW1},  an intuitively compelling wrong turn is described in \autoref{sec:covariant-input}, and the correction of that wrong turn is given in Sections \ref{sec:covariant-sec} and \ref{secdia}.  Using preliminaries from equivariant bundle theory given in \autoref{GBUND},  the functor $\Cpre$ is constructed in \autoref{sec:construction-cpre}.  The construction is summarized in the key diagram \autoref{eq:contra}. 

Operads and monads finally appear in \autoref{ContraIn}.  In \autoref{OperadCG}, which overlaps slightly with \cite{KMZ1}, we give a more detailed treatment of the new kind of equivariant operad introduced there,\footnote{These are very different from the $N_{\infty}$ equivariant operads of \cite{BP, GW, Rubin1, Rubin3}.} which we call a finitary $G$-operad.  It is not just an operad of $G$-spaces.  Rather,  it replaces the natural numbers by ordered finite $G$-sets.    However an operad $\sC$ of $G$-spaces prolongs to a $G$-operad $\NewC$, and the precise relationship between the two is at the core of the theory.  Using $\NewC$, we give the functor $\Cp$ a monad structure in \autoref{OGmonad} and discuss its algebras in \autoref{sec:cpre-algebra}.  The monad $\Cp$ that we construct in \autoref{OGmonad}\footnote{In fact, starting with a given operad, we construct from it  two quite different monads on orbital presheaves, as we explain in \autoref{sec:cpre-algebra}. The second one plays a conceptual role in \cite{KMZ1} but is irrelevant here.} 
 is equivalent to the one Costenoble--Waner constructed (see the discussion above \cite[Proposition 3.2]{CW}), but they construct it quite differently.

After discussing examples in Sections \ref{Examples}-\ref{PICBR}, we return to the general theory in \autoref{sec:presheaf-machine}, giving our Elmendorf construction in \autoref{sec:an-altern-elmend}, giving its key property in \autoref{sec:comp-natur-transf}, and going from there to the proof of \autoref{thm:pre-machine} in \autoref{sec:final}.  

To give motivation, recall that the functor $\bR$ has a left adjoint $\bL$ that sends an orbital presheaf $\sX$ to the $G$-space $\sX(G/e)$.   However, $\bL$ does not preserve weak equivalences.  Therefore one must replace $\bL$ by some functor $\generalE\colon \OGop \rtarr G\sT$ that comes with a natural weak equivalence $\bR\com \generalE \rtarr \Id$.  There are two very different known choices for $\generalE$.  One is given by the Elmendorf construction, here denoted $\originalE$ \cite{Elm}.  The other is given by observing that  $G\sO^{op}[\sT]$ and $G\sT$ have model structures such that the adjunction $(\bL,\bR)$ restricts to an adjoint {\em isomorphism} between the respective subcategories of cofibrant objects \cite{Pia, MM, EHCT}. Thus one can take $\generalE = \bL\com \GA$, where $\GA$ is cofibrant approximation in $G\sO^{op}[\sT]$.  

To construct the machine, one must obtain an analogous functor $\ourEonAlg$, as displayed in \autoref{ElmC}.  The hardest work in \cite{CW} is the construction of such a monadic variant of 
$\originalE$.  The construction there depends on the combinatorial structure of the operad $\sC$.  We instead construct a new variant, denoted $\ourE$, of the original Elmendorf construction. It is defined independently of $\sC$, but by standard use of a monadic two-sided bar construction, it extends to the required functor $\ourEonAlg$.  We expect these variants of the Elmendorf construction to have other applications.

The framework of \cite{KMZ1} more naturally sees the model theoretic alternative to the Elmendorf construction since the relevant adjunctions are central to the general framework.  The model theory works the same way for $\Cpa$ and $\bC[G\sT]$ as it does for $\OGop$ and $G\sT$, giving isomorphisms between the subcategories of cofibrant objects in each pair. Thus, alternatively, we can take $\sE_{\sC}$ to be $\bL\com \GA$, where $\GA$ is cofibrant approximation in $\Cpa$.
The recognition principle can be obtained either way, giving two proofs of the same theorem.  

\begin{notn}\label{notation}  For a monad $\mathbb{C}$ in a category $\sV$, we write $\mathbb{C}[\sV]$ for the category of $\mathbb{C}$-algebras in $\sV$.  For a small category $\sA$, we write $\sA[\sV]$ for the category of functors $\sA \rtarr  \sV$.
Notice that, for a second small category $\sB$, we have evident identifications
\begin{equation}\label{ABsee} \sA\big[\sB[\sV]\big] \iso (\sA\times \sB)[\sV]\iso \sB\big[\sA[\sV]\big].
\end{equation}
\end{notn}

\begin{rem}  The paper \cite{BO} of Bohmann and Osorno deals with more sophisticated categorical presheaf input with output in the Guillou-May model \cite{GM2} for $G$-spectra as presheaves of spectra. Like ours, it also gives Eilenberg-MacLane $G$-spectra.
Their work has been used to construct and study equivariant Waldhausen algebraic $K$-theory by Malkievich, Merling, and others \cite{MalMer1, MalMer2, CCM}. 
\end{rem}

\section{Orbital presheaves in the Grothendieck context}\label{GSPACE}

We establish context in this preliminary section.  
\subsection{The categories $\LRH$ and $\LRG_{G/H}$}\label{sec:categ-lambd-lambd}

We begin with some pedantically explicit elementary definitions.

\begin{defn}\label{LambdaG} For any finite group $G$, define $\LG$ to be the category whose objects are the finite based $G$-sets;  its morphisms are the injective based $G$-maps.  Each object has a basepoint $\ast$ and $n$ non-basepoint elements for some $n\geq 0$.  We shall later restrict to an equivalent small subcategory.  We view $\LG$ as a subcategory of the category $G\sF$ of finite based $G$-sets and all based $G$-maps between them.  

These definitions are standard, but the basepoint is both a necessity later and a distraction here when discussing finite sets with group actions.  We put a tilde over definitions to indicate the unbased or reduced versions.
Thus $\LRG$ denotes the category of (unbased) finite $G$-sets and $G$-maps, and so forth.  Adding disjoint basepoints gives the left adjoint of an adjoint equivalence between $\LRG$ and $\LG$. 

Define the orbit category $\OG$ to be the category of (unbased) orbits $G/H$ and $G$-maps and recall that a $G$-map between orbits is a surjection.  Embed $\OG$ as a subcategory of $G\sF$ by adjoining disjoint basepoints to orbits. 
\end{defn}

When $K$ is a subgroup of $H$, there is a restriction functor 
$\LRH \rtarr  \LRK$.  To assemble the $\LRH$ for all $H\subset G$ into a contravariant functor 
$\OG\rtarr \mathrm{Cat}$, we need an alternative description of $\LRH$ that remembers that
$H$ is a subgroup of $G$.  With $\sF$, $\sI$, and $\sO$ there taken to be $\FRG$, $\LRG$ and $\OG$ here, this also places us in the categorical context of \autoref{phistar}. 

\begin{defn}\label{LambdaGH} Define  $\LRG_{G/H}$ to be the category whose objects are (unbased) finite $G$-sets $T$ over $G/H$. Write $\pi\colon T\rtarr G/H$ generically for the projections.   Since $G$ acts transitively on $G/H$,  $\pi$ is a surjection. The  morphisms of  $\LRG_{G/H}$ are the injective $G$-maps $\ps\colon S\rtarr T$ over $G/H$.  Define
$$\mathrm{fib}(S) = \pi^{-1}(eH).$$
The action of $G$ on $S$ restricts to an action of $H$ on $\mathrm{fib}(S)$. An injective $G$-map $\ps\colon S\rtarr T$ over $G/H$ restricts to an injective $H$-map $\mathrm{fib}(S)\rtarr \mathrm{fib}(T)$, thus giving a ``fiber functor" 
 $$\mathrm{fib}\colon \LRG_{G/H} \rtarr \LRH.$$
\end{defn}

\begin{defn} The induction functor 
$$\mathrm{ind}\colon \LRH \rtarr \LRG_{G/H}$$ 
is defined by
$$\mathrm{ind}(S) = G\times_H S$$
with left $G$-action induced by the left action of $G$ on itself. The trivial $H$-map 
$S \rtarr \ast$ induces $\pi\colon G\times_H S \rtarr G\times_H \ast \iso  G/H.$
Thus $\pi$ sends $S$ to $eH$. 
\end{defn}

The following adjoint equivalence of categories formalizes the intuition that finite $H$-sets are equivalent to finite $G$-sets over $G/H$. 

\begin{prop}\label{prop:n-delta} The pair $(\mathrm{ind},\mathrm{fib})$ specifies an adjoint equivalence between 
the categories $\LRH$ and $\LRG_{G/H}$.
 \end{prop}
\begin{proof}
The unit $\et\colon S\rtarr \mathrm{fib}(G\times_H S)$ of the adjunction is the evident identification of $S$ with the fiber of 
$\pi\colon G\times_H S \rtarr G/H$.  The counit $\epz\colon G\times_H \mathrm{fib}(T) \rtarr T$ is the $G$-bijection of $G$-sets over $G/H$ induced by the action of $G$ on $T$. It can be seen to be a bijection by writing $T$ as a disjoint union of orbits and checking that $\epz$ maps orbits to orbits. 
\end{proof}

\begin{rem}\label{sense}  Consider a $G$-set  $T$ over $G/H$
with fiber $H$-set $S$.  When $H = G$, $S=T$. In general, if $S$ is a union of $n$ orbits $H/K_i$, where $K_i\subset H$, then 
$G/H = (G/K_i)/H$ for each $i$  and $T$ is the union of the $n$ orbits $G/K_i$.
\end{rem}

The following closely related induction functor will be needed later.

\begin{defn} \label{defn:Induce-from-H}Denote by $\bI_H: \LRG_{G/H} \rtarr \LRG$  the forgetful functor that sends $\pi: T \rtarr G/H$ to $T$. Under the equivalence $\LRH \simeq \LRG_{G/H}$, $\bI_H$ sends an $H$-set $S$ to the $G$-set $T=G\times_H S$.  Observe that $\bI_G$ is an identification since $\LRG_{G/G} = \LRG$. \end{defn}

For a $G$-map $\ph\colon G/K\rtarr G/H$, specialization of the categorical \autoref{phistar} defines 
$\ph^*\colon \LRG_{G/H} \rtarr \LRG_{G/K}$. It is given by pullback squares
\begin{equation}\label{reverseO}
\xymatrix{
\ph^*(T) \ar[r] \ar[d]_{\pi} & T \ar[d]^{\pi} \\ 
G/K \ar[r]_{\ph} &  G/H. \\}
\end{equation}

\begin{defn}\label{gammag} Recall that if $\ph(eK) = gH$, then the equality $\ph(eK) = \ph(kK)$ for $k\in K$ implies that 
$g^{-1}Kg\subset H$. Define $\ga_g\colon K\rtarr H$ to be the composite of conjugation, $c_{g^{-1}}(k) = g^{-1} k g$, and that inclusion. Restriction along $\ga_g$ defines a functor $\ga_g^*\colon \LRH \rtarr \LRK$. Note that $\ga_g$ depends on the choice of $g$ in its coset. We shall return to that ambiguity in \autoref{Orderings2}.
\end{defn} 

We shall use the $\ga_g^*$ to relate $\ph^*$ to the isomorphism of \autoref{prop:n-delta}. 

\begin{lem}\label{phistar2} For $H$-sets $S$, there is a natural isomorphism of $G$-sets over $G/K$
$$ G\times_K \ga_g^* S \rtarr \ph^*(\mathrm{ind}(S)) $$
and therefore a natural isomorphism of $K$-sets 
$$ \ga_g^* S \rtarr \mathrm{fib}(\ph^*(\mathrm{ind}(S))) $$
between their fibers.
\end{lem}
\begin{proof} Define $\io_g\colon \ga_g^*S \rtarr G\times_H S = \mathrm{ind}(S)$ by $\io_g(s) = (g,s)$.  Then $\io_G$ is a $K$-map since, for $k\in K$,
$$ \io_g(\ga_g(k) s) = (g, g^{-1}kg s) = (gg^{-1}kg,s) = (kg,s) = k(g,s).$$
Then $\io_g$ extends to a $G$-map $\overline{\io}_g$ that makes the outer square in the following diagram commute. 
\begin{equation*} 
\xymatrix{
G\times_K \ga_g^*S  \ar@{-->}[rd] \ar[rrd]^{\overline{\io}_g} \ar[rdd]_{\pi}& & \\
& \ph^*(G\times_H S)\ar[d]_{\pi} \ar[r] & G\times_H S \ar[d]^{\pi}\\
& G/K \ar[r]_-{\ph} &  G/H \\}
\end{equation*}
The universal property of the pullback square gives the dotted arrow.  It is a bijection.  Indeed, it is a surjection since $\pi$ and $\overline{\io}_g$ are surjections, and it is an injection since, by inspection, the source and target have the same cardinality. Application of $\mathrm{fib}$ gives the isomorphism of fibers.  
\end{proof}

 \subsection{Redefinitions in terms of ordered finite $G$-sets}
 \label{sec:categ-lambd-lambd2}
Observe that the category $\SI$ of symmetric groups is the subcategory of automorphisms in $\tilde{\LA}$.  It is essential to our work to build permutations explicitly into our categorical data. This is done by replacing the 
$\LRH$ and $\LRG_{G/H}$ by equivalent small categories of ordered finite $H$-sets and ordered finite $G$-sets over $G/H$. We shall again be very pedantic about this, in part following \cite{MMO}. 

 \begin{conv}\label{finiteGset}
Let $\mathbf{n}$ denote the (unbased) set $\{1, \dots, n\}$. For a finite group $G$ and a homomorphism $\al\colon G\rtarr \SI_n$, define $\bn^{\al}$ to be the $G$-set specified by letting $G$ act on $\mathbf{n}$ by $g\cdot i = \al(g)(i)$  for $1\leq i\leq n$.  Conversely, a $G$-action on $\mathbf{n}$ determines a $G$-homomorphism $\al$ by the same formula. A finite $G$-set $S$ with $n$ elements $\{s_1, \cdots, s_n\}$ can be identified with $\bn^{\al}$, where $\al(g)(i) = j$ if $gs_i = s_j$. We understand finite $G$-sets to be of this form henceforward,  and we shall often use the alternative notation $S_{\al}$ for  $S = \bn^{\al}$.  In later sections, we will add the disjoint $G$-fixed basepoint $0$  without change of notation.
\end{conv}

Explicitly, $\bn^{\al}$ is the subgroup of $\SI_n$ consisting of those permutations $\ga$ such that
$$ \ga \big((g,\al(g)(i)\big) = (g,\al(g))\ga(i) \ \ \text{for} \ \ 1\leq i \leq n.$$
Note in particular that if $\bn^{\alpha} \cong G/H$, then $\SI_{\bn^{\al}} \cong  W_G(H)= N_G(H)/H$.

\begin{lem}\label{lamblamb} If $S$ is specified by $\al\colon G\rtarr \SI_m$ and $S'$ is specified by $\al'\colon G\rtarr \SI_n$, then a $G$-map $\ph\colon S \rtarr S'$ is specified by a function $\ps\colon {\mathbf{m}} \rtarr \mathbf{n}$ such that 
\begin{equation}\label{albe}
\ps(\al(g)(i)) = \be(g)(\ps(i)) \ \ \text{for} \ \ 1\leq i\leq m.
 \end{equation}
\end{lem}
\begin{proof}
Given $\ph$, define $\ps$ by $\ps(i) = j$ if $\ph(s_i) = t_j$. Then \autoref{albe} expresses that $\ph$ is a $G$-map. Conversely, given $\ps$, define $\ph(s_i) = t_j$ if $\ps(i) = j$. 
\end{proof}

\begin{defn}\label{orders} Let $G$ have order $t_g$.  Fix a total ordering of $G$
 with $e=1$.  Let $\al_G\colon G\rtarr \SI_{t_g}$ be the homomorphism given by the
 action of $G$ on itself, and define $\al_H$ similarly using the ordering of $H$ as a
 subset of $G$. Let $t_{g/h} = |G/H|$, so that $t_{g/e} = t_g$; the letter $t$ is a reminder that $G$ acts transitively on each $G/H$. Order the $G$-set $G/H$ by choosing the minimum (in the ordering of $G$) of the elements of each coset $gH$ and then ordering these minima (again in the ordering of $G$); note that the coset $eH$ is numbered as $1$. Let $\al_{G/H}\colon G\rtarr \SI_{t_{g/h}}$ be the homomorphism specifying the action of $G$ on $G/H$.  
\end{defn}

\begin{defn}\label{redefine1} Redefine $\LRH$ to be the category of finite ordered $H$-sets $S = \bn^{\al}$, often denoted $S_{\al}$.   The morphisms are the  injective $H$-maps, namely those $\ps\colon \bm\rtarr \bn$ that satisfy \autoref{albe}.  It is clear that the new $\LRH$ is equivalent to the original one.  Similarly, redefine $\LRG_{G/H}$ as follows. Let $T$ be a finite $G$-set over 
$G/H$.  Then $T$ must have $t_{g/h} n$ elements for some $n\geq 1$, hence is specified by a homomorphism  $\be\colon G \rtarr \SI_{t_{g/h}n}$. Identifying $G/H$ with $\{1, \cdots, t_{g/h}\}$, we require $\pi\colon T\rtarr \{1, \cdots, t_{g/h}\} $ to satisfy
\begin{equation}\label{wreath1}
\pi(\be(g)(i)) = \al_{G/H}(g)(\pi(i)) \ \ \text{for} \ \ 1\leq i\leq t_{g/h} n
\end{equation}

We identify $\mathbf{t_{g/h} n}$ with $\mathbf{t_{g/h}} \times \mathbf{n}$ ordered lexicographically, thinking of the first copy of 
$\{\mathbf{1, \cdots, n}\}$ as corresponding to $\pi^{-1}(eH)$ and the remaining copies as corresponding to the $\pi^{-1}(g_iH)$, where the non-identity coset representatives $g_i$ are ordered by the chosen ordering of $G$. This fixes an injection 
$$\SI_{t_{g/h}} \times \SI_n \rtarr \SI_{t_{g/h}n}$$ 
Then $\be$ restricts to $\al\colon H\rtarr \SI_n$, where
$\SI_n$ acts on the first copy of $\{\mathbf{1, \cdots, n}\}$, and $\al$ determines $\be$ via the resulting identification of $T$ with the lexicographically ordered $G$-set $G\times_H \mathbf{n}^{\al}$ over $G/H$.  These ordered $T$ are the objects of our redefined category $\LRG_{G/H}$. The morphisms are the injective $G$-maps over $G/H$. It is clear that the new $\LRG_{G/H}$ is equivalent to the original one.  
\end{defn}

\begin{rem}\label{sense1} If we decompose  a  finite $G$-set $T$ over $G/H$ as a coproduct of orbits  
$S_i = G/K_i$, say $n$ of them, then $\pi$ ensures that each $K_i$ is subconjugate to $H$.  We can choose our decomposition so that $K_i\subset H$.  Then  $t_{g/k_i} = t_{g/h} n_i$ and $\sum_i n_i  = n$.  The ordering of $G$ fixes an ordering of the cosets in each  $G/K_i$ consistent with the ordering of the cosets in $G/H$, and the ordering of $T$ fixes an ordering of the orbits.  
\end{rem}

From here, the rest of \autoref{sec:categ-lambd-lambd} carries over directly with $\LRH$ and $\LRG_{G/H}$ interpreted in our new sense.  We slightly expand on this blanket assertion in the following remarks.

\begin{rem}\label{Orderings}  We have defined our ordered objects of $\LRH$ and $\LRG_{G/H}$ so that the ordering of $G$ and of a finite $H$-set $S$ determines the ordering of $\mathrm{ind}(S)$, and the ordering of a $G$-set $T$ over $G/H$ determines the ordering of  $\mathrm{fib}(T)$.  Similarly, for $\ph\colon G/K \rtarr G/H$, $\ph^*(T)$ in \autoref{reverse} inherits a lexicographical ordering from the ordering of $G/H$ and a given ordering of $T$.  That is, $(gK,t) < (g'K,t')$ if $g< g'$ or if $g= g'$ and $t< t'$; $(eK,1)$ is the first element. Here $\ph(gK) = \pi(t)$ and $\ph(g'K) = \pi(t')$.  
\end{rem}

\begin{rem}\label{Orderings2} For an ordered $H$-set $S$, we have an implied ordering of the targets of the displayed isomorphisms in \autoref{phistar2}, and that fixes orderings of their sources $G\times_K \ga_g^*S$ with $\ga_g^{*}$ as defined in \autoref{gammag}.  This works for any choice of $g$, but we can fix the choice by requiring $g$ to be minimal in its coset $gH$.  With this choice, $g = e$ when $\ph$ is induced by an inclusion $K\subset H$.  Note that if $S$ is defined by $\al\colon H\rtarr \SI_n$, then $\ga_g^*S$ is defined by the composite $\al\com c_{g^{-1}}\colon K \rtarr \SI_n$. 
\end{rem}

\begin{notns}\label{GrothGcats} Henceforward, we always understand the $\LRH$ and $\LRG_{G/H}$ in the ordered sense. Letting $H$ vary and using \autoref{phistar} with $\sF = \LRG$ and $\sO = G\sO$, we  obtain Grothendieck categories $\LRG_{\star}$ and $\LRG_{\star}^{vop}$ via Definitions \ref{LambdaStar} and \ref{LambdaStarv}. Adding disjoint basepoints to all orbits, all $H$-sets, and all $G$-sets over $G/H$, it is straightforward to reinterpret everything in a based sense, giving equivalent categories $\LH$ and 
$\LG_{G/H}$, the latter consisting of based $G$-sets over $(G/H)_+$ with $\pi^{-1}(\ast) = \ast$.  {\em {With this understanding, we use the notations $\LH$ and $\LG_{G/H}$, without the tilde, for the rest of the paper, even though when dealing with $H$-sets or $G$-sets explicitly, we understand them in the unbased sense.   We use the notation $\bn$ for both the finite set $\{1,\dots, n\}$ and the based finite set $\{0, 1, \cdots, n\}$.}}
\end{notns}

\subsection{Equivariant preliminaries}\label{sec:equiv}

It is convenient to put together in one place some conventions on $G$-sets and $H$-sets that we shall use.   The following notations are variants of some in 
\cite{MMO} and expansions of those in \cite{KMZ1}.   We let $H$ be a subgroup of $G$, but we mainly focus on $H=G$.

\begin{notn}\label{graphhom} For a homomorphism $\al\colon H\rtarr \SI_n$, let 
$$\GA_{\al} = \{(h,\al(h))\}\subset G\times \SI_n$$
denote the graph subgroup determined by $\al$. Note that the projection $\pi\colon \GA_{\al} \rtarr H$ is an isomorphism.
\end{notn}

\begin{notn}\label{Yalpha} Let $Y$ be a $(G\times \PI)$-space for some group $\PI$, such as $\PI = \SI_n$, and let
$\al\colon H\rtarr \PI$ be a homomorphism. We define $Y^\al$ to be the $H$-space with underlying space $Y$ and with a new $H$-action $\cdot_\al$ given by
$$h\cdot_\al y= (h,\al(h))\cdot y.$$  
Thus $Y^{\al}$ is $Y$ with $H$-action twisted by $\al$. In other words, $Y^{\al}$ is $Y$ with its action by the subgroup $\GA_{\al}$ of $G\times \Pi$, pulled back along $\pi^{-1}\colon H\rtarr \GA_{\al}$. 
\end{notn}

Inspection gives the following result.  
\begin{lem}\label{standard}
Let $S=\bn^{\al}$, $\al\colon H\rtarr \SI_n$.  If $Y$ is a based $(G\times \SI_n)$-space, then
\begin{equation*}\label{twistfix}
  Y^{\Gamma_{\al}} \iso (Y^{\al})^H.
\end{equation*}
If $X$ is a based $G$-space and $F_H(S,X)$ is the space of based $H$-maps $S\rtarr X$, this specializes to 
\begin{equation*}\label{eq:first} 
(X^n)^{\GA_{\al}} \iso F_{H}(S, X).
\end{equation*}
In particular, taking $H=G$ and taking  $S = G/K = t_{g/k}^{\al_{G/K}}$, 
\begin{equation}\label{YesYes}
(X^{\al_{G/K}})^G \iso X^K.
\end{equation}
\end{lem}

We will not use semi-direct products until \autoref{ContraIn} and wreath products will only appear implicitly, where their use allows a more explicit  description of $\bC^{pre}$ than we shall give; see \autoref{OGmonad}.

\begin{defn}\label{wreath}  Let  $G$ and $\PI$ be groups and let $\ph\colon G\rtarr \mathrm{Aut}(\PI)$ be a homomorphism of groups.  The semi-direct product
$\PI\rtimes_{\ph} G$ is the group given by the set $\PI\times G$ with unit $(1,1)$ and product
$$  (\si,h)(\ta,g) = (\si \ph(h)(\ta),hg). $$
For a homomorphism $\al\colon G\rtarr \SI_n$, define $\ph(\al)\colon  G\rtarr \PI^n$ by 
$$\ph(\al)(g)(\pi_1, \cdots, \pi_n) = (\pi_{\al(g)(1)}, \cdots, \pi_{\al(g)(n)})$$
Then define the wreath product $\PI\wr_{\al} G$ to be the semi-direct product $\PI^n\rtimes_{\ph(\al)} G$.
\end{defn}

\begin{notn}\label{LAG} Recall that $G\LA$ is the category of finite based $G$-sets and based $G$-injections.
Co-ambiguously,\footnote{Boardman's joke.} we let both $\LA_G$ and $\ul{G\LA}$ denote the category of finite based $G$-sets and all based injections, with $G$ acting by conjugation on the morphism sets. Note  that $G\LA = (\ul{G\LA})^G$.  We usually use the notation $\LA_G$ in other papers, such as 
\cite{KMZ1}.  As a rule of thumb, we prefer $\ul{G\LA}$ when thinking of the self-enrichment of the category $G\LA$ and $\LA_G$ when thinking of $\LA_G$ as a stand-alone $G$-category, where a $G$-category is a category enriched in $G$-sets.
\end{notn}

\begin{conv}\label{finiteGset2} Let $\SI_G$ be the $G$-subcategory of bijections in $\LA_G$. 
Just as $\SI_n = \SI(\bn,\bn)$, we define  $\SI_{\bn^{\al}}$ to be the group $\SI(\bn^{\al}, \bn^{\al})$ of automorphisms of $\bn^{\al}$ in $\SI_G$.   It is the group $\SI_n$ equipped with the action of $G$ given by conjugation by $\al$.  Explicitly, $\SI_{\bn^{\al}}$ is the subgroup of $\SI_n$ consisting of those permutations 
$\ga$ such that
$$ \ga \big((g,\al(g)(i)\big) = (g,\al(g))\ga(i) \ \ \text{for} \ \ 1\leq i \leq n.$$
Note in particular that if $\bn^{\alpha} \cong G/H$, then $\SI_{\bn^{\al}} \cong  W_G(H)= N_G(H)/H$.
\end{conv}

\begin{lem}\label{semi}
Define $\al_c\colon G\rtarr  Aut(\SI_n)$ to be conjugation by $\al$, 
$$ \al_c(g)(\si) = \al(g)\com \si \com \al(g)^{-1}. $$
Then $\al_c$ is a homomorphism of groups.  The action of $G$ on $\bn^{\al}$ extends to an action of the semi-direct product  
$\SI_n \rtimes_{\al_c} G$ via
$$  (g,\si)(i) = \si\big(\al(g)(i)\big).$$
\end{lem}
\begin{proof} For the first statement, it is easily checked that
$$\al_c(g)(\si \ta) = \al_c(g)(\si)\al_c(g)(\ta)\ \ \text{and} \ \ \al_c(gh) = \al_c(g)\big(\al_c(h)\big).$$
Thus each $\al_c(g)$ is an automorphism of $\SI_n$  and $\al_c$ is a group homomorphism.  For the second statement, $\SI_n \rtimes_{\al_c} G$ is the set $\SI_n\times G$ with the product
$$  (\si,g)(\ta,h) = (\si\al_c(g)(\ta), gh).$$
It is easily checked that 
$$\big((\si,g)(\ta,h))(i)  = (\si,g)\big((\ta,h)(i)\big).  \qedhere $$
\end{proof} 

\section{Outline of the construction of  $\Cp$}\label{CW1}

\subsection{Inputs to the categorical tensor product $\sCprest\otimes_{G\LA_{\star}} \bS_{\star}$}\label{Cpre1} 

The monad in $G\sT$ associated to an operad $\sC$ of $G$-spaces is the categorical tensor product defined on a based $G$-space $X$ by

\begin{equation}\label{oldC}
\mathbb{C}X = \sC \otimes_{\Lambda} \bQ X, 
\end{equation}
where $\LA$ is the category of finite based sets $\bn =\{0, 1, \cdots, n\}$, $n\geq 0$, and based injections.  Here $\sC$ is seen as a contravariant functor $\LA \rtarr G\sU$, where $G\sU$ is the category of unbased $G$-spaces, by neglect of structure (see \autoref{LAtoLAG}) and $\bQ X$  is the covariant functor $\LA \rtarr G\sT$ defined by
$$ \bQ X(n) = X^n. $$
Functoriality on $\LA$ is obtained using basepoint inclusions \cite{MayGeo, MZZ}, and the assumption that $\sC(0)$ is a point gives that $\bC X$ inherits a basepoint from that of $X$.   Observe that 
\begin{equation}\label{Qprod}
\bQ X(m+n) = X^{m+n} = X^m\times X^n = QX(m) \times QX(n). 
\end{equation}
That is crucial in determining the monad structure on  $\bC X$.

Ignoring the monad structure, $\bC$ is a functor $G\sT\rtarr G\sT$ obtained by composing $\bQ\colon G\sT \rtarr \LA[G\sT] $ with the more general functor
$$\hat{\bC}\colon \LA[G\sT] \rtarr G\sT, \ \  \text{where} \ \ \hat{\bC}\sX = \sC\otimes_{\LA} \sX$$ 
for orbital presheaves $\sX$.  Thus $\bC X = \hat{\bC} QX$. Although our interest is in operads, $\sC$ here can be any functor $\sC\colon \LA^{op} \rtarr G\sT$.  It is convenient to insist that $\sC(0)=\{\ast\}$. 

We construct $\Cp\colon \OGop \rtarr \OGop$ analogously and in comparable generality.  We return to the monad structure in \autoref{OGmonad}.  The construction is motivated by the following result, which will put us in a conceptual framework for recognition principles for composite adjunctions that is developed in  \cite{KMZ1}.  It will be proven in \autoref{sec:construction-cpre}. 

\begin{prop}\label{prop:phi-C-commute} The following diagram commutes up to natural isomorphism when 
$\bC$ and $\Cp$ are constructed from a given functor $\sC\colon \LA^{op} \rtarr G\sU$. \begin{equation*} 
  \begin{tikzcd}
    G\sT \ar[r, "\bC"] \ar[d, "\bR"'] & G \sT \ar[d, "\bR"] \\
    \OGop \ar[r,"\Cp"'] & \OGop
  \end{tikzcd}
\end{equation*}
\end{prop}

We define the functor $\Cp$ as a ``fiberwise categorical tensor product''.  For an orbital presheaf $\sX$,
\begin{equation}\label{newC0} 
\Cp\sX = \sCprest\otimes_{G\LA_{\star}} \bS_{\star}(\sX)
\end{equation}
is defined by specializing \autoref{const:tensorOverLA}.  We describe it in the next section, after explaining $S_{\star}$ and $\sCprest$ here.
The definition starts from the categorical Definitions \ref{defn:fiberOversG} and \ref{defn:cofiberOversG}, together with \autoref{defn:fiberOversGbis}; these specialize to define covariant and contravariant ``Grothendieck  $\OG$-functors".  

In  \autoref{Gcats}, we use $G$-categories and $G$-functors to give a general way to construct Grothendieck $\OG$-functors from  $G$-functors that occur in nature.  Informally, $G$-categories are categories with compatible actions of $G$ on Hom objects and $G$-functors are functors that are equivariant on Hom objects.  We define fixed-point functors,  both denoted $\ul{R}_{\star}$, from appropriate $G$-categories to covariant and contravariant $\OG$-categories.  A functor $\sC\colon \LA^{op} \rtarr G\sU$, where $G\sU$ is the category of unbased $G$-spaces and $G$-maps, prolongs to a $G$-functor $\NewC\colon G\LA^{op}\rtarr G\sU$, and we have the contravariant $\OG$-functor
\begin{equation}\label{sCpre}
\sCprest = \ul{\bR}_{\star}\NewC.
\end{equation}
 
In \autoref{sec:covariant-input}, we construct intuitively compelling composite functors
$$\bP_{G/H}\bQ_{G/H}\colon \LG_{G/H} \rtarr  \sT. $$
and show that these give a covariant Grothendieck $\OG$-functor 
$\bP_{\star}\bQ_{\star}$.  Here $\bP_{\star}$ is obtained by  prolongation from functors to $G$-functors and $\bQ_{\star}$ is obtained by power operations. This construction is interesting and gives one
 route to a monad on $G\sO^{op}[\sT]$, namely the composite monad $\bR\bC\bL$, where $\bL$ is left adjoint to 
 $\bR$.  That is {\em{not}}  the monad we are after.  In \autoref{sec:covariant-sec}, we construct different functors 
 $$\bS_{G/H} \colon \LG_{G/H} \rtarr  \sT. $$ 
 These are defined from sections of $G$-functors, and they give a covariant Grothendieck $\OG$-functor  $\bS_{\star}$. In analogy with \autoref{Qprod}, we will have 
 \begin{equation}\label{Sprod}
( \bS_{G/H}\sX)(S\amalg T) = (\bS_{G/H}\sX)(S)\times (\bS_{G/H}\sX)(T)
 \end{equation} 
 for an orbital presheaf $\sX$ and $G$-sets $S$ and $T$ over $G/H$.  This gives the route to the monad $\Cp = \hat{\bC}^{pre}\com \bS_{\star}$ that we want.
 
 \subsection{A more explicit description of $\Cp$}\label{Cpre2}

Let $\sX$ be an orbital presheaf.  On orbits $G/H$,  $\Cp \sX$ is given by
\begin{equation}\label{newC1}
(\Cp \sX)(G/H) =   \sCpre_{G/H} \otimes_{\LG_{G/H}} \bS_{G/H}(\sX).
 \end{equation}
Explicitly, this is the coequalizer exhibited in the diagram
\begin{equation}\label{newC2}
\xymatrix@1{\coprod_{(S,T)}  \sCpre_{G/H}(T)\times \LG_{G/H}(S,T) \times (\bS_{G/H}\sX)(S) 
\ar@<-.7ex>[d] \ar@<.7ex>[d] \\
\coprod_{T}   \sCpre_{G/H}(T)\times (\bS_{G/H}\sX)(T) \ar[d] \\
  \sCpre_{G/H} \otimes_{\LG_{G/H}} \bS_{G/H}\sX}
\end{equation}
Here $(S,T)$ runs over pairs of $G$-sets over $G/H$ and the parallel arrows are disjoint unions of maps induced by the respective action maps
 $$ \sCpre_{G/H}(T)\times \LG_{G/H}(S,T) \rtarr \sC_{G/H}(S)$$
 and  
 $$\LG_{G/H}(S,T)\times (\bS_{G/H}\sX)(S) \rtarr  (\bS_{G/H}\sX)(T).$$
 A $G$-map $\ph\colon G/K\rtarr G/H$ induces a functor $\ph^*\colon \LG_{G/H}\rtarr \LG_{G/K}$.  To define $\Cp \sX$ on morphisms, in Sections  \ref{Gcats} and \ref{sec:covariant-sec} we construct functors 
\begin{equation}\label{Cphi} \sCpre_{\ph}\colon \sCpre_{G/H} \rtarr \sCpre_{G/K}\com (\ph^*)^{op}
\end{equation}
and
\begin{equation}\label{Sphi}
\bS_{\ph}\sX\colon \bS_{G/H}\sX \rtarr \bS_{G/K}\sX \com \ph^*
\end{equation}
that give contravariant and covariant Grothendieck $\OG$-functors $\sC^{pre}_{\star}$ and $\bS_{\star}$.   Then \autoref{prop:fiberOversG2} kicks in to convert these to functors on the Grothendieck categories 
$G\LA^{vop} _{\star}$ and $G\LA_{\star}$. 
Expressed in terms of the respective categorical action maps, for $G$-sets $S$ and $T$ over $G/H$, these are given by maps
$$ \xymatrix{ \sCpre_{G/H}(T)\times \LG_{G/H}(S,T) \ar[d]   & \text{and} &  \LG_{G/H}(S,T) 
\times (\bS_{G/H}\sX)(S)\ar[d]\\
\sC_{G/K}(\ph^*S) & &  (\bS_{G/K}\sX)(\ph^*T) \\} $$
These maps and the morphism maps $\ph^*\colon \LG_{G/H}(S,T) \rtarr \LG_{G/K}(\ph^*S,\ph^*T)$ induce maps of coequalizers
$$ \ph^* = \sCpre_{\ph} \otimes_{\ph^*}\bS_{\ph}(\sX) \colon
(\Cp \sX)(G/H) \rtarr (\Cp\sX)(G/K)  $$
that complete the construction of $\Cp \sX$ as an orbital presheaf. 

We show in \autoref{OGmonad}  that the operad structure on $\sC$ leads to the definition of the structure maps exhibiting $\Cp$ as a monad in $G\sO^{op}[\sT]$. We can replace $\bS_{\star}$ by $\bP_{\star}\bQ_{\star}$ in the argument, but that results in the monad  $\bR\bC\bL$.  

\section{The input Grothendieck $\OG$-functors}

\subsection{$\ul{\bR}_{\star}\colon \ul{G\LA}[\ul{G\sT}] \rtarr \LG_{\star}[\sT]$ and $\ul{\bR}_{\star}\colon \ul{G\LA}^{op} [\ul{G\sU}] \rtarr \LG_{\star}^{op}[\sU]$}\label{Gcats}

Recall that $G\sU$ is the category of $G$-spaces and $G$-maps and $G\sT$ is the category of based $G$-spaces and based $G$-maps.  Following \cite{MMO}, we write $\ul{G\sU}$ for the self-enrichment of $G\sU$. It is the category of $G$-spaces and continuous maps, with $G$ acting by conjugation on the Hom function spaces.   We define $\ul{G\sT}$ similarly, using based function spaces $F(X,Y)$, the basepoint being the trivial map $X \to {\ast} \to Y$. Since $F(X,Y)^G = G\sT(X,Y)$, we can view $G\sT$ as the fixed point category $(\underline{G\sT})^G$.  

We define $\ul{G\LA}[\ul{G\sT}]$ to be the category of $G$-functors  
$\sY\colon\ul{G\LA}  \rtarr \ul{G\sT}$.  This means that $\sY$ assigns a based $G$-space $\sY(S)$ to each finite based $G$-set $S$ and assigns $G$-maps from the hom $G$-sets of based injections $S\rtarr T$ to the hom $G$-spaces $F(\sY(S),\sY(T))$.  Thus $\ul{G\LA}[\ul{G\sT}]$ is an enriched analogue of the category $\LA[G\sT]$  of functors $\LA\rtarr G\sT$. Without basepoints, we have the similar enriched category $\ul{G\LA}[\ul{G\sU}]$. 

\begin{defn}\label{ulRstar} The functor $\ul{\bR}_{\star}\colon   \ul{G\LA}[\ul{G\sT}] \rtarr \LG_{\star}[\sT]$ is the levelwise fixed-point-space functor given on $\sZ\colon \ul{G\LA} \rtarr \ul{G\sT}$ by
\begin{equation} \label{eq:22}
  (\ul{\bR}_{G/H}\sZ)(T) = \sZ(\bI_HT)^{H}
\end{equation}
where $T$ is a $G$-set over $G/H$ and $\bI_H$ is the forgetful functor defined in \autoref{defn:Induce-from-H}.  Both vertical functoriality on $G$-injections $T\rtarr T'$ over $G/H$ and horizontal functoriality on $G$-maps $G/K\rtarr G/H$ are immediate, so that we have a covariant Grothendieck $\OG$-functor $\ul{\bR}_{\star}\sZ$ in $\LG_{\star}[\sT]$. 
\end{defn}

The contravariant analogue is precisely similar, with notation chosen to mesh with the operadic examples to come later.

\begin{defn}\label{ulRstarcon} The functor $\ul{\bR}_{\star}\colon \ul{G\LA}^{op} [\ul{G\sU}] \rtarr \LG_{\star}^{op}[\sU]$ is the levelwise fixed-point-space functor given on 
$\NewC\colon  \ul{G\LA}^{op} \rtarr \ul{G\sU}$ by
\begin{equation}\label{eq:33}
(\ul{\bR}_{G/H}\NewC)(T) = \NewC(\bI_H T)^H
\end{equation}
where $T$ is a $G$-set over $G/H$.  It is again an immediate verification that we have a contravariant Grothendieck $\OG$-functor 
$\ul{\bR}_{\star}\NewC$ in $\LG_{\star}^{op}[\sU]$.  In more detail, for a $G$-map $\ph\colon G/K\rtarr G/H$ and a map $f\colon S\rtarr T$ of $G$-sets over $G/H$, we obtain
an induced map $\NewC(\bI_H T)^H \rtarr  \NewC(\bI_K \ph^*(T))^K$ by the contravariant functoriality of $\NewC$ and of passage to fixed points.\footnote{The diagrams in \autoref{phistar} may help give the picture.}  With $\sCprest = \ul{R}\NewC$, this gives the functors $\sCpre_{\ph}$  promised in \autoref{Cphi}. 
\end{defn}
  
\subsection{The covariant input functors $\bQ_{\star}$ and $\bP_{\star}$}
 \label{sec:covariant-input}
 
We here define covariant Grothen\-dieck $\OG$-functors 
$\bQ_{\star}$ and $\bP_{\star}$ and construct  a commutative diagram
\begin{equation} \label{eq:covariant}
\xymatrix{
G\sT \ar[r]^-{\bQ} \ar[d]_{\bR}  & \LA[G\sT] \ar[r]^-{\bP} \ar[d]^{\bR_{\star}} & \ul{G\LA}[\ul{G\sT}] \ar[d]^{\ul{\bR}_{\star}}\\
\OGop \ar[r]_-{\bQ_{\star}}   & \big(\LA \times G\sO^{op}\big)[\sT]  \ar[r]_-{\bP_{\star}} &\LG_{\star}[\sT] \\}
\end{equation}
The bottom row is {\em{not}} what we will use in defining $\Cp$, although it seems to be dictated by comparison with $\bC$.  Its consideration will lead us to the right substitute $\bS_{\star}$ in \autoref{sec:covariant-sec}.  Remember that, for a $G$-space $X$,
\begin{equation}\label{QR} 
\bR(X)(G/H) = X^H \ \ \text{and} \ \ \bQ(X)(\bn) = X^n.
\end{equation} 
The bottom middle category in \autoref{eq:covariant} is the category of functors $\LA \times G\sO^{op} \rtarr \sT$, and we have the identifications of categories from \autoref{ABsee}:
$$G\sO^{op}\big[\LA[\sT]\big] \iso \big(\LA \times G\sO^{op}\big)[\sT] \iso \LA\big[G\sO^{op}[\sT]\big].$$
\begin{defn}\label{RQstar} The functor $\bQ_{\star}$ in \autoref{eq:covariant} is the levelwise power functor given on $\sX\colon G\sO^{op} \rtarr \sT$ by
\begin{equation}\label{Qstar}
(\bQ_{\star}\sX)(n,G/H) = \sX(G/H)^n.
\end{equation}
The functor $\bR_{\star}$ is the levelwise fixed-point functor given on $\sY\colon \LA \rtarr G\sT$ by
\begin{equation}\label{Rstar}
(\bR_{\star}\sY)(n,G/H) = \sY(n)^H.
\end{equation}
\end{defn}

\begin{lem} The left square in \autoref{eq:covariant} commutes.
\end{lem}
\begin{proof} For a $G$-space $X$, $(X^n)^H$ can be identified with $(X^H)^n$, respecting functoriality on $\LA\times G\sO^{op}$.
\end{proof}

The following result gives the prolongation functor, denoted $\bP$ in \autoref{eq:covariant}. The $G$-trivial $G$-sets $\bn^{\epz_n}$ give an inclusion $\LA\subset G\LA$, and $\LA[\ul{G\sT}] = \LA[G\sT]$ since $G$ acts trivially on $\LA$, so that a $G$-functor $\LA\rtarr \ul{G\sT}$ takes values in $G$-fixed hom spaces.

\begin{prop}\label{PUone}  The forgetful functor  $\bU\colon \ul{G\LA}[\ul{G\sT}] \rtarr \LA[\ul{G\sT}] = \LA[G\sT]$ has a left adjoint prolongation functor $\bP$, and $(\bP,\bU)$ is an adjoint equivalence. 
\end{prop}
\begin{proof}  This is implied by the formal part of \cite[Theorem  2.37]{MMO}.   For $\sY$ in 
$\LA[\ul{G\sT}]$,  $\bP \sY$ is given by the categorical tensor product
\begin{equation}\label{PUtwo}
(\bP \sY)(\bn^{\al})  = \ul{G\LA}(-, \bn^{\al})\otimes_{\LA} \sY \iso \sY(\bn)^{\al}.
\end{equation}
Since $\bU$ is full and faithful, the unit $\et\colon \sY \rtarr \bU\bP \sY$ is the evident identification.  For $\sZ\in \LA_G[\ul{G\sT}]$, the counit $\epz\colon \bP \bU \sZ \rtarr \sZ$ is an isomorphism by a Yoneda type verification using the identity function on $\bn$, regarded as an element of $\LA_G (\bn,\bn^{\al})$.
\end{proof}

\begin{rem}\label{contraP}  With $\sT$ replaced by $\sU$ and covariant functors replaced by contravariant functors, we have a precisely analogous adjunction $(\bP,\bU)$ relating 
$\LA^{op}[G\sT]$ and $\ul{G\LA}^{op}[\ul{G\sT}]$.  For $\bC \in \LA^{op}[G\sT]$, we write $\bC_G =\bP\sC$ and take
$$\NewC(\bn^{\al})  =  \sC(\bn)^{\al}.$$
This will play an important role later, when we consider operads.
\end{rem}

The following definition completes the construction of the diagram \autoref{eq:covariant}.

\begin{defn}  Let $\bL_{\star}$ be the left adjoint of $\bR_{\star}$.  It is given explicitly by 
$$(\bL_{\star}\sY)(\bn) = \sY(\bn,G/e).$$  
Then define $\bP_{\star}  = \ul{\bR}_{\star} \bP \bL_{\star}$.  Visibly, $\bL_{\star}\bR_{\star} = \Id$ and therefore $\bP_{\star}\bR_{\star} =  \ul{\bR}_{\star} \bP$.
\end{defn}

\begin{con}\label{humbug} Similarly, $\bL\bR = \Id$, where $\bL$ is the left adjoint of $\bR$.  Therefore we can replace $\bQ_{\star}$ by $\bR_{\star}\bQ\bL$ and still have the left square commute. Use of these substitutes instead of the $\bS_{\star}$ of \autoref{sec:covariant-sec} leads by an easy formal argument  in \cite{KMZ1}  to the monad $\bR\bC\bL$ on the category of orbital presheaves.  
\end{con}

The  monad $\bR\bC\bL$ plays a conceptual role in the theory of \cite{KMZ1}, but it  is not the Costenoble-Waner monad that we seek.  There are other wrong choices to be left in decent obscurity.  We turn to the choice that works.

\subsection{The section covariant input functor $\bS_{\star}$}\label{sec:covariant-sec}
We construct a covariant Grothen\-dieck functor $\bS_{\star}$ here and prove in the next section that the following replacement of the diagram \autoref{eq:covariant} commutes. 

\begin{equation} \label{eq:covariant2}
\xymatrix{
G\sT \ar[r]^-{\bQ} \ar[d]_{\bR}  & \LA[G\sT] \ar[r]^-{\bP}  & \ul{G\LA}[\ul{G\sT}] \ar[d]^{\ul{\bR}_{\star}}\\
\OGop    \ar[rr]_{\bS_{\star} } &  &\LG_{\star}[\sT] \\}
\end{equation}
We have replaced the bottom row $\bP_{\star} \bQ_{\star}$ of \autoref{eq:covariant} with the single functor $\bS_{\star}$,  which does not factor through $(\LA\times G\sO^{op})[\sT]$.
As we shall see, the definition of $\bS_G$ is dictated by \autoref{Sprod} and its analog for $\ul{\bR}_G\com \bP\com \bQ$. Intuitively, these force us to define $\bS_G$ on orbits and then extend it by additivity, taking disjoint unions of finite $G$-sets to products of based spaces.  We then extend $\bS_G$ to $\bS_{\star}$ by composing $\bS_G$ with the forgetful functors $\bI_H$ of \autoref{defn:Induce-from-H}.  

\begin{defn}\label{SecDefn} Let $\sX$ be an orbital presheaf and $S=\bn^{\al}$ be a finite (unbased) $G$-set.  For $1\leq i\leq n$, let $G_i =\{g|gi = i\}$ denote the stabilizer of $i$.   We first define an action of $G$ on
$\amalg_{i=1}^n\sX({G/G_i})$ and then use that to define $\bS_G(\sX) \colon \LG\rtarr \sT$. 

For $g\in G$, suppose 
$gi\equiv \al(g)(i) =j$. Then $gG_ig^{-1} = G_j$.  Left multiplication by $g^{-1}$ specifies a $G$-map of orbits
 $$g^{-1}\colon  G/G_j \rtarr G/g^{-1}G_jg = G/G_i.$$
These maps satisfy $(g'g)^{-1} = g^{-1}(g')^{-1}$ and thus define an action of $G$ on $\amalg_{i=1}^n G/G_i$.  By contravariant functoriality, application of $\sX$ to the maps ${g^{-1}}$ gives maps 
\begin{equation}\label{gee}
g\colon \sX(G/G_i) \rtarr \sX(G/G_{j})
\end{equation}
that specify the required (left) action of $G$ on $\amalg_{i=1}^n \sX(G/G_i)$.

Define a $G$-map
\begin{equation*}
\rh\colon \amalg_{i=1}^n\sX({G/G_i}) \rtarr  S
\end{equation*}
by sending $\sX({G/G_i})$ to $i$.  A section of $\rh$ is a $G$-map 
$$  \si\colon S \rtarr \amalg_{i=1}^n \sX(G/G_i)$$
such that $\rh\com \si = \id$.   Define
$$ \bS_G(\sX)(S) = \mathrm{Sec}_G\big(S,\amalg_{i=1}^n \sX(G/G_i)\big) $$ 
to be the space of all such sections.  It is based, with basepoint given by the unique section $\si$ such that $\si(s)$ is the basepoint in its image based space for all $s\in S$.

For $G$-sets $S$ and $T$, we order $S\amalg T$ in the evident way, first $S$ and then $T$. It follows directly that $\bS_G$ converts disjoint unions to products as in \autoref{Sprod}.   Thus if $S=S_1\amalg \cdots \amalg  S_n$ is a decomposition of $S$ as a disjoint union of orbits, then
\begin{equation*}
(\bS_G\sX)(S) \iso \times_{i=1}^n (\bS_G\sX)(S_i). 
\end{equation*}
When $n=1$, so that $S = G/H$ for some $H$, we have $n=t_{g/h}$.  A section is then determined by the image of the identity coset in $\sX(G/H$), which is arbitrary.  Indeed,  given $x\in \sX(G/H)$, there is a unique $G$-map 
$G/H\rtarr X$ that sends $eH$ to $x$ since $\si$ must send the coset $g_iH$ to  $g_i x \in \sX(G/H)$.  This gives that
$(\bS_G \sX)(G/H) = \sX(G/H)$ and therefore

\begin{equation}\label{SecWed}
(\bS_G\sX)(S) \iso \times_{i=1}^n \sX(S_i). 
\end{equation}

We could take this as an alternative definition of $\bS_G\sX$, except that the decomposition of $S$ as a disjoint union of orbits is not canonical.  It is unique up to isomorphism, hence by specialization of the following details, that alternative is well-defined up to isomorphism.   We must define $\bS_G(\sX)(\ps)$ on $G$-injections 
$\ps\colon S=\bm^{\al} \rtarr \mathbf{n}^{\beta} = T$, such as isomorphisms.   Let $K_i$ be the stabilizer of 
$i \in \mathbf{m}^{\alpha}$ and $H_j$ be the stabilizer of $j \in \mathbf{n}^{\beta}$.  
Then we can identify $S$ with $\amalg_i G/G_i$ and $T$ with $\amalg_j G/G_j$.  We have $K_i = H_j$ 
when $j=\ps(i)$.  Then  $\ps$ maps $G/G_i$ isomorphically onto $G/G_j$ and the functoriality of $\sX$ gives 
homeomorphisms $\sX(G/G_i)\iso \sX(G/G_j)$. 
Inserting the basepoint in the coordinates not in the image of $\ps$, these maps assemble to a
$G$-equivariant map
$$ (\bS\sX)(S) = \prod_{i=1}^m \sX(G/K_i) \rtarr  \prod_{j=1}^n\sX(G/H_j)= (\bS\sX)(T)$$
over $\ps$, and this construction is functorial.
\end{defn}

\begin{defn}\label{SstarDefn} Let $\sX$ be an orbital presheaf. We define $\bS_{\star}\sX \colon \LG_{\star} \rtarr \sT$. We first define $\bS_{G/H}\sX\colon \LG_{G/H} \rtarr \sT$.  For a $G$-set $T$ over $G/H$, define
\begin{equation}\label{secGH}
(\bS_{G/H}\sX)(T) = (\bS_G\sX)(\bI_H T)
\end{equation}
where $\bI_H$ is the forgetful functor of \autoref{defn:Induce-from-H}.
As in \autoref{defn:fiberOversG}(ii), we must define compatible natural transformations 
$$(\bS_{G/H}\sX)_{\ph}\colon  (\bS_{G/H}\sX) \rtarr (\bS_{G/K}\sX)\com \ph^*$$
for $G$-maps $\ph\colon G/K\rtarr G/H$, as promised in \autoref{Sphi}.  We can rewrite this as
$$(\bS_{G}\sX)_{\ph}\colon  (\bS_{G}\sX)\circ \bI_H \rtarr (\bS_{G}\sX)\circ \bI_K\com \ph^*.$$
For a $G$-set $T=(\mathbf{t_{g/h}n})^{\beta}$ over $G/H$, we can identify $\ph^*T$ as a $G$-set 
$\mathbf{t_{g/k}n}^{\ga}$ over $G/K$, where $\ga$ is constructed using $\be$ and $\ph$.  
Then the pullback $\ph^*(T)\rtarr T$
over $\ph$ of \autoref{reverse} induces (as $\sX$ is contravariant)
$$ \amalg_{i=1}^{t_{g/h}n} \sX(G/ G_i) \rtarr \amalg_{j=1}^{t_{g/k}n} \sX(G/ G_j) 
$$
This induces the map on sections
$$
\mathrm{Sec}_G(T,\amalg_{i=1}^{t_{g/h}n} \sX(G/ G_i)) \rtarr 
\mathrm{Sec}_G(\ph^*T,\amalg_{j=1}^{t_{g/k}n} \sX(G/ G_j),
$$
which is defined to be the map $(\bS_G\sX)_{\ph}\colon (\bS_G\sX)(T) \rtarr (\bS_G\sX)(\ph^*T)$.  Verifying its conditions,
\autoref{defn:fiberOversG} shows that we have constructed a functor $\bS_{\star}\sX \in \LG_{\star}[\sT]$. 
\end{defn}

\subsection{The proof that $\bS_{\star}\com \bR$ can be identified with $\ul{\bR}_{\star} \com \bP\com \bQ$}\label{secdia}With definitions in place, we can now prove that the diagram \autoref{eq:covariant2} commutes, starting with the following partial result in that direction.

\begin{lem}\label{dumbPeter} The composites $\ul{\bR}_G\com \bP\com \bQ$ and $\bS_G\com \bR$  
can be identified.
\end{lem}
\begin{proof} Let $X$ be a $G$-space. First consider an orbit $G/K= \mathbf{t}^{\al}$, where $t = t_{g/k}$ and 
$\al=\al_{G/K}$.  Here the definitions and  \autoref{YesYes} give the first equality and  the isomorphism in
\begin{equation}\label{Yeah}
(\bR_G \bP\bQ)(X)(G/K) =({(X^t)^{\al}})^G \iso X^K =  (\bS_G\bR)( X)(G/K).
\end{equation}
The second equality is part of the definition of $\bS_G$. 

Since the functor $(\bS_G\bR)( X)$ takes disjoint unions to products, it remains to check that the functor $(\bR_G\bP\bQ)( X)$ also does so.  Thus let $S = \bm^{\al}$ and $T = \bn^{\be}$.  Then $S\amalg T = (\bm+\bn)^{\al\times \be}$ where 
$\al+ \be$ denotes the evident composite
$$\xymatrix@1{ G \ar[r]^-{\DE} & G\times G \ar[r]^-{\al\times \be} & \SI_m\times \SI_n \ar[r]^-{\subset} & \SI_{m+n}.\\}$$
A quick inspection shows that
$$ (X^{m+n})^{\GA_{\al+\be}} = (X^m)^{\GA_{\al}}\times (X^n)^{\GA_{\be}}  \qedhere $$
\end{proof}

\begin{lem}   The composites $\ul{\bR}_{\star}\com \bP\com \bQ$ and  $\bS_{\star} \com \bR$ can be identified.
\end{lem}
\begin{proof}
We first observe that $\bR_{G/G} \bP  \bQ \iso \bS_{G/G}  \bR$, since that is just a restatement of \autoref{dumbPeter}.
Next, we show that, as functors $\LG_{G/H} \rtarr \sT,$
\begin{equation}
\label{eq:commute-at-H}
(\bR_{G/H}\bP\bQ)( X)\cong (\bS_{G/H}\bR)( X).
\end{equation}
By \autoref{secGH}, $\bS_{G/H}\sX = \bS_G\sX \circ \bI_H$.  In general, $\bR_{G/H}\sY \neq \bR_G\sY \circ \bI_H$, but we claim that equality up to canonical isomorphism does hold when $\sY = \bP\bQ X$:
\begin{equation}
  \label{eq:commute-at-H-lemma2}
(\bR_{G/H}\bP\bQ)( X) \cong (\bR_{G}\bP\bQ)( X) \circ \bI_H.
\end{equation}
Combined with $\bR_G\com \bP\com \bQ \cong \bS_G \com\bR$ and
\autoref{secGH} with $\sX = \bQ X$, this proves \autoref{eq:commute-at-H}.
The claim \autoref{eq:commute-at-H-lemma2} follows from the fact that, for a $G$-set
$T = (\mathbf{tn})^{\be}$ over $G/H$ with fiber over $eH$ the $H$-set
$\mathbf{n}^{\al}$, we have by \autoref{standard} and inspection that
\begin{equation*}
  (\bR_{G/H}\bP\bQ)(X)(T)  = (\bQ X)(n)^{\Gamma_{\al}} = (X^n)^{\Gamma_{\al}},
\end{equation*}
\begin{equation*}
  (\bR_{G}\bQ)(X) \circ \bI_H) (T)
                        = (\bR_{G} \bQ)( X)((\mathbf{tn})^{\be}) =
                         (X^{tn})^{\Gamma_{\be}},
\end{equation*}
and
\begin{equation*}
 (X^{tn})^{\Gamma_{\be}} \cong \mathrm{Map}_G((\mathbf{tn})^{\be}, X)
  \cong \mathrm{Map}_H(\mathbf{n}^{\al}, X) \cong (X^n)^{\Gamma_{\al}}.
\end{equation*}

The proof of compatibility of \autoref{eq:commute-at-H} across fibers, that is with respect to the natural transformations of \autoref{defn:fiberOversG}(ii), is similar. 
\end{proof}

\section{The functor $\Cpre$ and the proof that $\bR\com \bC\iso \Cp\com \bR$}\label{RvsR}

\subsection{A key diagram} \label{sec:construction-cpre}  
Now return to our outline in \autoref{CW1}.  Let $\NewC$ be a $G$-functor $\ul{G\LA}^{op}\rtarr \ul{G\sU}$ such that $\NewC(0 )= \{\ast\}$.  We may assume without loss of generality that $\NewC$ is the prolongation of its restriction $\sC\colon \LA^{op} \rtarr G\sU$, which in practice is the underlying functor of a (reduced) operad $\sC$ in $G\sU$.
We then have the functor  $\hat{\bC}\colon \LA[G\sT] \rtarr G\sT$, and we define $\hat{\bC}_G \colon \ul{G\LA}[\ul{G\sT}] \rtarr G\sT$ by 
$$\hat{\bC}_G \sZ = \NewC\otimes_{\ul{G\LA}} \sZ.$$
   For $\sY \colon G\LA_{\star} \rtarr \sT$, we define
\begin{equation}\label{Cpredefns} 
\sCpre_{\star} = \ul{\bR}_{\star}\NewC = \ul{\bR}_{\star}\bP\sC \ \ \text{and} \ \  \Cpre \sY = \sCpre_{\star} \otimes_{G\LA_{\star}} \sY.
\end{equation} 
Remember that this is constructed orbitwise as
\begin{equation}
  \label{def:Cpre_action2}
  \CpreGH\sY = \sC^{pre}_{G/H}\otimes_{\LG_{G/H}} \sY_{G/H}.
\end{equation} 
This gives the functor $\Cpre$ in the following extension of  \autoref{eq:covariant2}.  
\begin{equation}\label{eq:contra} 
\xymatrix{
G\sT \ar[r]^-{\bQ} \ar[d]_{\bR}& \LA[G\sT] \ar[r]^-{\bP} \ar[dr]_{\ul{\bR}_{\star}\bP}
\ar@/^2pc/ [rr]^-{\hat{\bC}}
& \ul{G\LA} [\ul{G\sT}] \ar[r]^-{\hat{\bC}_G} \ar[d]^{\ul{\bR}_{\star}} & G\sT \ar[d]^{\bR} \\
\OGop \ar[rr]_-{\bS_{\star}} & &  G\LA_{\star}[\sT] \ar[r]_-{\Cpre}  & \OGop \\}
\end{equation}
We claim that the diagram commutes up to natural isomorphism.  The composites in the rows are the functors denoted $\bC$ and $\Cp$ in the diagram of \autoref{prop:phi-C-commute}.  Therefore the claim will complete the proof of that result. The left trapezoid is from \autoref{eq:covariant2} and is irrelevant to our work in this subsection; the tautological triangle is all we need.  After the next result, we will focus on the square at the right. 


The following identification in \autoref{eq:contra} came as a pleasant surprise.

\begin{prop}\label{cuter}  The functor $\hat{\bC}$ can be identified with the composite $\hat{\bC}_G\com \bP$. 
\end{prop}
\begin{proof}
We must show that, when $\sY = \bP \sX$ for some $\sX\in \LA[G\sT]$, the restriction of the coequalizer 
defining $\hat{\bC}_G \sY$ to trivial $G$-sets $\bn =\bn^{\epz_n}$ gives an identification with $\hat{\bC}\sX$.  The coequalizer in general is 
{\small{
$$ 
\xymatrix@1{
\coprod_{\bm^{\al},\bn^{\be}} \hat\sC_G(\bn^{\be})\times \LA_G(\bm^{\al},\bn^{\be})\times \sY(\bm^{\al}) 
\ar@<.7ex>[r]  \ar@<-.7ex>[r] & \coprod_{\bp^{\ga}} \hat\sC_G(\bp^{\ga}) \times \sY(\bp^{\ga})\ar[r] & \hat\bC_G(\sY).  \\}  
$$
}}
Restricting to a single relevant domain component, we see the pair of maps
$$ \xymatrix@1{
\hat\sC_G(\bn^{\al}) \times \LA_G(\bn^{\epz},\bn^{\al})\times \sY(\bn^{\epz}) \ar@<.7ex>[r]  \ar@<-.7ex>[r]   & \hat\sC_G(\bn^{\epz})\times \sY(\bn^{\epz})  \coprod \hat\sC_G(\bn^{\al})\times \sY(\bn^{\al}). \\}
$$
Taking the identity map of $\bn$ as an element of $\LA_G(\bn^{\epz},\bn^{\al})$, this tells us that the images of $\hat\sC_G(\bn^{\al})\times \sY(\bn^{\al})$ and $\hat\sC_G(\bn^{\epz})\times \sY(\bn^{\epz})$ are identified in $\hat\bC_G \sY$.  But when $\sY=\bP \sX$ 
$$\hat\sC_G(\bn^{\epz})\times \sY(\bn^{\epz}) =\sC(n) \times \sX(n).  \qedhere$$
\end{proof}

An analogous proof addresses a question about the right square.  The definition of $\ul{\bR}_{G/H}$ in  \autoref{eq:33}  implicitly uses all $H$-sets $\bn^{\al}$, whereas $\bR$ sees only the restriction to $H$ of $G$-sets $\bn^{\al}$.  We illuminate the distinction with the following definition, in which we implicitly compare the cases $G/G$ and
$G/H$ of \autoref{defn:Induce-from-H}  and \autoref{eq:33}. 

\begin{defn} \label{CpreDefn}   For $\sC\colon G\LA^{op} \rtarr \sT$ and $H\subset G$, define $(\NewC)^H\colon H\LA^{op} \rtarr \sT$ by 
\begin{equation}\label{CGthree1}
(\NewC)^H(\bn^{\al}) = (\sC(n)^{\al})^H
\end{equation}
for $\al\colon H\rtarr \SI_n$.  
Define restriction 
$$\mathrm{res}\colon (\NewC)^H\rtarr  \sCpreGH $$ 
by restricting to those $\al\colon H\rtarr \SI_n$ that are restrictions to $H$ of some $\al\colon G\rtarr \SI_n$. Focusing on $H\LA \iso G\LA_{G/H}$, define
\begin{equation}\label{CGthree2}
\hat{\bC}_G^H\sY = (\NewC)^H\otimes_{H\LA}\sY_{H}.
\end{equation}
Here $\sY\in G\LA_{\star}[\sT]$ and $\sY_{H}$ denotes its component $\sY_{G/H}$, viewed as a functor $H\LA \rtarr \sT$ via the cited isomorphism.
\end{defn}

With these notations, the following identification is another pleasant surprise.

\begin{prop}\label{cuter2} Restriction induces an identification of  $\hat{\bC}_G^H\sY$ with $\CpreGH\sY$.
\end{prop}
\begin{proof}
The proof is analogous to that of \autoref{cuter}. For an $H$-set $S = \bn^{\al}$, let 
$T=(\mathbf{t_{g/h}n})^{\be}$ be the $G$-set $G_+\sma_H S$ and let $\io\colon S\rtarr \mathrm{res}(T)$ be the $H$-injection of its fiber.  Restricting to a relevant domain component in the coequalizer diagram that defines $\hat{\bC}_G^H\sY$, we see the pair of maps
$$ \xymatrix{
\big((\NewC)^H(\mathrm{res}(T)) \times H\LA(S,\mathrm{res}(T))\times \sY(S) \ar@<.7ex>[d]  \ar@<-.7ex>[d] \\  
\big((\NewC)^H(S)\times \sY(S)\big)  \coprod 
\big({\sC}_G^H(\mathrm{res}(T))\times \sY(\mathrm{res}(T))\big). \\}
$$
Taking $\io$ in $H\LA(S,\mathrm{res}T)$, we see that the image of 
$(\NewC)^H(S)\times \sY(S)$ in $\hat{C}_{H}^{G}\sY$ is identified with the image of 
$(\NewC)^H(\mathrm{res}(T))\times \sY(\mathrm{res}(T))$.  Since  $(\NewC)^H(\mathrm{res}(T))$ is $\sCpreGH(T)$, this says that all of $\hat{\bC}_G^H\sY$
is in the image of $\Cpre_{G/H}\sY$. 
\end{proof}

More substantially, the right square of \autoref{eq:contra} encodes a commutation of coequalizers with passage to fixed points that is not at all obvious a priori.  We address this using a bit of equivariant bundle theory that is recalled in the following subsection.

\begin{defn} We say that $\sC$ is $\SI$-free if each $\sC(n)$ is a free $\SI_n$-space.
\end{defn}

\begin{prop}\label{RCdiagram} If $\sC$ is $\SI$-free, then the right square in the diagram \autoref{eq:contra} commutes up to natural isomorphism.
\end{prop}
\begin{proof} By the universal property of the domain coequalizer, for each $H\subset G$, we have a natural map 
$$ \om\colon (\Cpre\ul{\bR}_{\star})(\sY)(G/H) \rtarr (\bR\hat{\bC}_G)(\sY)(G/H). $$
We claim that $\om$ is an isomorphism for each $H$, proving the result.  We first observe that it suffices to prove that $\om$ is an isomorphism after precomposition with $\bP$.  Indeed, by \autoref{PUone}, we then have
$$ \bR\hat{\bC}_G\sY \iso \bR\hat{\bC}_G\bP\bU\sY \iso \Cpre\bR_{\star}\bP\bU\sY \iso \Cpre\bR_{\star}\sY$$
for $\sY\in \ul{G\LA}[\ul{G\sT}]$.  That allows us to use \autoref{cuter} to replace $\hat {\bC}_G\com \bP$ with $\hat{\bC}$.  Thus we replace $\sY$ by $\bP\sX$, where $\sX \in \LA[G\sT]$. 

Morphisms in $H\LA$ are compositions of automorphisms and ordered proper inclusions of $H$-sets.  As is standard for $\otimes_{\LA}$, $\otimes_{H\LA}$ can be computed in two steps.  We first use automorphisms to pass to orbits, and then use the equivalence relation ($\sim$) obtained using the ordered injections.  For the composite $\bR\hat {\bC}$, we carry out the first step in \autoref{prop:fiberbundle} below, which implies that
\begin{equation}\label{ThisHelps}
(\bR\hat{\bC}\sX)(G/H) = \coprod_{n\geq 0}  \coprod_{[ \al] \in \mathrm{Rep}(H, \SI_n)}\sC(n)^{\Gamma_\al} \times_{\SI_{\bn^{\al}}} \sX(n)^{\GA_{\al}}/(\sim).
\end{equation}
In effect, this already commutes coequalizers past fixed points. By inspection, the definition of a coequalizer together with the definitions of $\bP$, $\bR_{\star}$, and $\sCpre$ show that, up to notation, $(\Cpre\bR_{\star}\bP\sX)(G/H)$ admits the same description.
\end{proof}

\subsection{A bit of equivariant bundle theory}\label{GBUND}
We here prove a basic relationship between passage to orbits and passage to fixed points.  Fix $G$ and $n\geq1$.  Let $C$ be a $\SI_n$-free ($G\times \SI_n$)-space and $Y$ be a 
$(G\times \SI_n)$-space.  We give a description of $(C\times_{\SI_n} Y)^H$ for $H\subset G$ that ties in to our construction of $\Cp$. Write
 \begin{equation*}
   \mathrm{Rep}(G, \Sigma_n) =\{\mathrm{homomorphisms ~~} G \rtarr
\Sigma_n\}/\Sigma_n\text{-conjugation}.
\end{equation*}
We have three sets in bijective correspondence:
\begin{equation*}
  \begin{array}{c|c}
    \text{ Elements } & \text{ Sets }  \\ \hline
    \alpha &\{\mathrm{homomorphisms ~~} G \to \Sigma_n\} \\
    \bn^{\alpha} & \{G\text{-sets of order $n$} \}\\
    \Gamma_{\alpha} & \{\text{ subgroups } \Gamma \subset G \times \Sigma_n \text{ such that }
    \Gamma \cap \SI_n = e\}
  \end{array}
\end{equation*}

Using the notations of Conventions \ref{finiteGset} and \ref{finiteGset2}, we sketch the proof of the following bundle-theoretic result.

\begin{prop}
\label{prop:fiberbundle} For $H\subset G$ and for a $\SI_n$-free $(G\times \SI_n$)-space $C$ (with $G$ acting on the left and $\SI_n$ acting on the right) and a (left)  $(G\times \SI_n$)-space $X$,
\begin{equation*}
 (C \times_{\Sigma_n} X)^H \iso
 \coprod_{[ \al] \in \mathrm{Rep}(H, \SI_n)}C^{\Gamma_\al} \times_{\SI_{\bn^{\al}}} X^{\GA_{\al}}.
\end{equation*}
\end{prop}

We see this by applying the following result about equivariant principal bundles to the principal  $G$-$\SI_n$-bundle $C\times X \rtarr C\times_{\SI_n}X$. The result is a special case of \cite[Theorem 12]{LM86} that is spelled out in \cite[Theorem 2.46 and Remark 2.47]{Zou}.

\begin{prop}
  \label{prop:bundle}Let 
$$ p\colon E\rtarr E/\SI_n = B$$
be a principal $G$-$\SI_n$-bundle.\footnote{Here and below, we tacitly assume that total spaces $E$ are completely regular.} Then for $H\subset G$
$$ B^H \iso \coprod_{[ \al] \in \mathrm{Rep}(H, \SI_n)} E^{\GA_{\al}}/\SI_{\bn^{\al}}.$$
\end{prop}
\begin{proof}[Sketch proof] Let$B_0$ be a component of $B^H$. There is some $\al\colon H \rtarr \SI_n$ such that $(p^{-1}(B_0))^{\GA_{\al}}$ is nonempty, and then $(p^{-1}(B_0))^{\GA_{\al'}}$ is nonempty if and only if $[\al'] = [\al]$ in $\mathrm{Rep}(H,\SI_n)$.  This gives a function
$$c: \pi_0(B^H)  \rtarr \mathrm{Rep}(H, \Sigma_n).$$
It need be neither injective nor surjective, and the summand indexed on 
$[\al]$ is empty if $[\al]$ is not in the image of $c$. For 
$[\alpha] \in \mathrm{Rep}(H, \Sigma_n)$, write
$B^H_{[\alpha]}$ for the $[\alpha]$-indexed component,
namely the pullback
\begin{equation*}
  \begin{tikzcd}
    B^H_{[\alpha]} \ar[r] \ar[d] &
    B^H \ar[d]\\
    {c^{-1}([\alpha])} \ar[r] & \pi_0(B^H)
  \end{tikzcd}
\end{equation*}
The principal $\SI_n$-bundle $p|_{B^H_{[\alpha]}}$ has a reduction of its structure group to $\SI_{\bn^{\al}}$, giving a principal $\SI_{\bn^{\al}}$-bundle 
\begin{equation*}
E^{\Gamma_\alpha} \rtarr B^H_{[\alpha]}.
\end{equation*}
Since $B^H$ is the disjoint union of the $B^H_{[\al]}$, the conclusion follows.
\end{proof}

\section{$G$-operads $\NewC$ and $\OG$-operads $\sC_{\star}$}\label{ContraIn}

\subsection{$G$-operads $\NewC$}\label{OperadCG} So far we have used operads only implicitly, as precursors of monads.
We will not repeat the full definition.  The following recollection focuses on the equivariant special case.

\begin{defn}\label{operad}  Recall that an operad $\sC$ in $G\sU$ consists of unbased $G$G-spaces $\sC(n)$, where $\sC(n)$ has a left action by $G$ and a right action by $\SI_n$ that commute with each other, together with a unit $G$-map $\id\colon  \ast \rtarr \sC(1)$ and structure 
$G$-maps  
$$\ga\colon \sC(k) \times \sC(j_1) \times \cdots \times \sC(j_i) \rtarr \sC(j),$$
where $j = j_1 + \cdots + j_k$, which are associative, unital, and equivariant as specified in \cite{MayGeo, MayOp1}.   We say that $\sC$ is reduced if $\sC(0) = \ast$, and we restrict attention to reduced operads henceforward. 
\end{defn}  

Motivated by a conceptual gap in the comparison between the equivariant operadic and Segalic infinite loop space machines, we defined a new kind of equivariant operad in \cite{KMZ1}, which we call a $G$-operad.  It only later became apparent that $G$-operads are actually more central to the orbital presheaf infinite loop space machine that is the subject of this paper. We therefore present more details here.  They may at first seem to have come out of nowhere, but they were dictated by the conceptual context of equivariant categories of operators presented in \cite{KMZ1}.\footnote{As a result, there is some duplication of material here and in \cite{KMZ1}.}

The idea of $G$-operads is to use  $G$-spaces $\NewC(\bn^{\al})$ rather than just the $G$-spaces $\sC(n)$ of an operad in $G\sU$. The details of the generalization are all about equivariance, and we shall use the details in \autoref{sec:equiv}.

\begin{defn}\label{composable} Consider tuples of finite $G$-sets 
$$(\bk^{\be};\bj_1^{\al_1}, \cdots, \bj_k^{\al_k})$$
We say that $(\be; \{\al_r\})$ is \em{composable} if  $\bj_r^{\al_r}  = \bj_s^{\al_s}$ whenever 
$\be(g)(r)= s$ for some $g\in G$, where $r,s\in \{1, \cdots, k\}$.  When this holds, define 
$$\ga(\be;\wed \al_r)\colon G\rtarr \SI_{j}$$ 
by letting 
$$\ga(\be;\wed \al_{r})(g)(i) = \al_s(g)(i) $$ 
when $\be(g)(r) = s$ for $1\leq r\leq k$; on the left, $i\in\{1, \cdots, k\}$ is in the $r$th ordered block of $j_r$ letters and, on the right, $i$ is in the $s$th ordered block of  $j_r =j_s$ letters.
\end{defn}

 
We can now give our definition of $G$-operads $\NewC$, using notations from the end of  \autoref{sec:equiv}.

\begin{defn}\label{crude}  A $G$-operad $\NewC$ in $G\sU$  consists of (unbased) $G$-spaces  $\sC(\bn^{\al})$ 
 for all $\al\colon G\rtarr \SI_n$.  We require
$\sC(\bn^{\al})$ to have a left action by $G$ and a right action by $\SI_n$ that together give a right action of the semi-direct product  
$\SI_n \rtimes_{\al_c} G$. We require a unit $G$-map $\id \colon \ast \rtarr \NewC(\mathbf{1})$ and structure 
$G$-maps  
$$\ga_G\colon \NewC(\bk^{\be}) \times \NewC(\bj_1^{\al_1} ) \times \cdots \times \NewC(\bj_k^{\al_k}) \rtarr 
\NewC(\bj^{\ga(\be; \wed \al_r)})$$
for composable $(\be,\{\al_r\}_{1\leq r\leq k})$.  We require the $\ga_G$ to be 
$$\SI_{k} \rtimes_{\be_c}  G \ \ \text{and} \ \  (\SI_{j_1} \times \cdots \times
  \SI_{j_k})\rtimes_{(\al_1,...,\al_k)_c} G$$ equivariant, where these groups act on the source of $\ga$ via their actions on its coordinates and where they act on the target via embeddings as subgroups of  $\SI_{j}$ via block permutations and permutations within blocks. The $\ga_G$ must be associative and unital for composable data, as specified in \cite[Definition 1]{MayOp1}.\footnote{In that definition, assuming the horizontal data are composable in diagram (a), one of the right vertical $\gamma$ is composable if and only if the other is.  Diagram (b) is always composable. In diagram (c), the left map is composable if and only if the right map is composable.} 
Again, we say that $\NewC$ is reduced if $\NewC(\mathbf{0}) = \ast$, and we restrict attention to reduced $G$-operads. 
\end{defn} 

\begin{lem}\label{LAtoLAG}  An operad $\sC$ restricts to a functor $\sC\colon \LA^{op} \rtarr G\sU$, and a $G$-operad $\NewC$ restricts to a $G$-functor $\NewC\colon \ul{G\LA}^{op} \rtarr \ul{G\sU}$.
\end{lem}  
\begin{proof} The first statement goes back to \cite[Construction 2.4]{MayGeo}, which is recalled in \cite[Definition 9.4]{KMZ1}.  The argument for $G$-operads is similar.   For an ordered injection $\ps\colon \bm^{\al} \rtarr \bn^{\be}$ (not a $G$-map), define 
$\ps^*\colon \sC(\bn^{\be}) \rtarr \sC(\bm^{\al})$ by 
$$ \ps^*(c) = \ga(c; \epz_1, \cdots, \epz_k),$$
where $\epz_r = \id\in\NewC(\mathbf{1})$ if $r$ is in $\mathrm{Im}(\ps)$ and $\epz_r = \ast \in \sC(0)$ if not.  Functoriality follows by use of composability and checks of equivariance. For a more conceptual proof based on comparison with categories of operators, see \cite[Lemma 7.45]{KMZ1}. 
\end{proof} 

\begin{prop}\label{CtoCG}  A $G$-operad $\NewC$ in $G\sU$ restricts on $G$-trivial based finite $G$-sets 
$\bn^{\epz_n}$ to an operad  $\sC = \bU\NewC$ in $G\sU$.  Conversely, an operad $\sC$ in $G\sU$ prolongs to a $G$-operad $\NewC = \bP \sC$ in $G\sU$.  The adjoint equivalence $(\bP,\bU)$ of \autoref{contraP} induces an adjoint equivalence between the category of operads in $G\sU$ and the category of $G$-operads in $G\sU$. 
\end{prop}
\begin{proof}  The prolongation starts from the definition $\NewC(\bn^{\al}) = \sC(n)^{{\al}}$. Ignoring the actions of $G$ on our finite sets, the structure maps are given by the structure maps of $\sC$. We have defined $G$-operads so that the structure maps $\ga$ are $G$-maps.   This result also appears as \cite[Proposition 7.30]{KMZ1}, where it is compared with an analog for categories of operators.  \end{proof}

\subsection{The monads of  $G$-spaces associated to $G$-operads}\label{Gmonad}

We now turn to the monads associated to operads $\sC$ and $G$-operads $\NewC$.  Fix an operad $\sC$ in $G\sU$ and a $G$-operad $\NewC$, which we can take to be $\bP\sC$.  We have already defined the underlying functor $\bC$ as $\hat{\bC}\com \bQ$, and it is standard that  the structure maps of an operad induce a monad structure on $\bC$.  Similarly, we define the underlying functor $\bC_G$ to be  $\hat{\bC}_G\com \bP\com \bQ$.   Then \autoref{cuter} shows that $\bC_G$ can be identified with $\bC$ and so inherits a monad structure.   Thus, although $G$-operads are new structures, they do not give rise to new monads.   However, that path is not available in the specification of  $\Cp = \Cpre\com \bS_{\star}$ as a monad.  For that, we must look a bit more carefully at the definitions.

\begin{rem}   Looking only at $\bC_G$, it seems that the product $\mu\colon \bC_G\bC_G \rtarr \bC_G$ might only be partially defined in view of the composability requirement in the definition of $\ga$.   However, \autoref{cuter} saves the day: the identification of $\bC_G$ with $\bC$ tells us how to get around this. \end{rem}

\subsection{$\OG$-operads $\sC_{\star}$ and monads of orbital presheaves}\label{OGmonad}

Since $G$-operads are defined in terms of products of $G$-spaces and passage to fixed points commutes with products, it is clear that the fixed points of $G$-operads $\NewC$ give a functor from $G\sO$ to operads of some sort.  We call these $G\sO$-operads or ``orbital operads" .  The only examples we know are of the form $\ul{\bR}_{\star}\NewC$, where $\NewC \iso \bP\sC$ for an operad $\sC$ in $G$-spaces.  We shall make these examples explicit starting from \autoref{ThisHelps}, but we first give some heuristics.  

Consider the outer rectangle of \autoref{eq:contra}.  Since $\bR \bC = \Cp \bR$,  $\bR\bC\bC =\Cp\Cp \bR$.  The monad $(\bC,\mu,\et)$ is defined in terms of commutative diagrams, and these diagrams still commute when post-composed with $\bR$.  Via our equalities, it follows that we have monad type diagrams for $\Cp$ after pre-composition with $\bR$. We therefore know that we have the structure  of a monad restricted to orbital presheaves of the form $\sX = \bR  X$ for a based $G$-space $X$.  Making this explicit is an exercise   from                  \autoref{ThisHelps}.  Replacing $X^H$  by $\sX(G/H)$ in the result tells us how to extend the  monad to a monad structure on $\Cp$, and this idea makes clear  why, although in principle we are starting from $\NewC$, its partial operadic structure disappears on passage to fixed points.  However, in terms of exposition, it seems simpler to just start with $\Cp$, taking  $\sCprest = \ul{\bR}_{\star}\bP\sC$ for a $\SI$-free operad $\sC$ of $G$-spaces.

Using \autoref{SecWed}, \autoref{secGH}, and \autoref{CpreDefn}, we have the expansion of definitions
\begin{eqnarray*}\label{unravel}
 (\Cp_{\star} \sX)(G/H) \ \ = \ \ \Cp_{G/H}\sX & =&  \coprod_T \sC^{pre}_{G/H}(T) \times (\bS_{G/H}\sX)(T) / {\approx} \\
         & =&  \coprod_T \sC^{pre}_{G/H}(T) \times_{\SI_T} (\bS_{G/H}\sX)(T)/{\sim}. \\
         & = & \coprod_T\sC^{pre}_{G/H}(T) \times_{\SI_T}  \prod_{i=1}^m \sX(G/K_i)/{\sim} \\
\end{eqnarray*}
Here $T$ runs over the finite $G$-sets over $G/H$ and $T = \amalg_{i=1}^m S_i$ is a decomposition into disjoint unions of orbits
 $S_i = G/K_i$ with $K_i$ subconjugate  to $H$.   Taking $T = (\mathbf{t}_{g/h}\bn)^{\be}$ as in \autoref{redefine1},  we have
$S_i = \mathbf{t}_{g/k_i} ^{\al_{G/K_i}} $, where $t_{g/k_i} = t_{g/h} n_i$ and $\sum_i n_i = n$.  We can view $\be$ as the composite

$$ \xymatrix@1{
G \ar[r]^-{\DE} & G^m \ar[rr]^-{\times_i\al_{G/K_i}}  && \times_i  \SI_{t_{g/k_i}} \ar[r]  & \SI_{{\tiny{\sum}}_i t_{g/k_i}} \ar[r]^{=}  & \SI_{t_{g/h}n}. \\}$$
Just as nonequivariantly, we first use equivariance relations and then basepoint identifications in the construction.  Specializing and expanding,  

\begin{eqnarray*} \Cp_{G/H}\Cp_{\star} \sX & =& \coprod_T \sC^{pre}_{G/H}(T) \times_{\SI_T} (\bS_{G/H}\Cp_{\star} \sX)(T)  / {\sim} \\
                                        & =&  \coprod_T \sC^{pre}_{G/H}(T) \times_{\SI_T} \prod_{i=1}^m(\Cp_{\star} \sX)(G/K_i)  / {\sim} \\
                       & =&   \coprod_T \sC^{pre}_{G/H}(T) \times_{\SI_T}  \prod_{i=1}^m \left(\sC^{pre}_{G/K_i}(S_i)\times_{\SI_{S_i}} \prod_{j=1}^{n_{i}} \sX(G/K_{i,j})  \right) / {\sim} \\
\end{eqnarray*}
Permuting variables, this is 
$$ \coprod_T\left(\sC^{pre}_{G/H}(T) \times_{\SI_T} \prod_{i=1}^m \sC^{pre}_{G/K_i}(S_i) \times_{\prod_{i=1}^m \SI_{S_i}}
                       \prod_{j=1}^{n_{i}} \sX(G/K_{i,j})      \right)/{\sim}.$$
Here $S_i = \amalg_{j=1}^{n_{i,j}}S_{i,j}$ is a decomposition into unions of orbits
$S_{i,j} = {G/K_{i,j}}$ with $K_{i,j}$ subconjugate to $K_i$.   Writing $S_i =(\mathbf{t}_{g/k_i} \bn_i)^{\be_i}$, $1\leq i\leq m$, we have  
$S_{i,j} = \mathbf{t}_{g/k_{i,j}}^{\al_{G/K_{i,j}}}$, where $t_{g/k_{i,j}} = t_{g/k_i}n_{i,j}$ and $\sum_j n_{i,j} = n_i$.  In view of \autoref{cuter2}, we get the same functors here if we replace  $\sC_{G/H}^{pre}$ and $\sC_{G/K{_i}}^{pre}$by $(\NewC)^H$ and 
$(\NewC)^{K{_i}}$.  Then we can pass to fixed points from the operadic structure maps of $\NewC$ to obtain maps
$$  (\NewC)^H(T) \times\prod_{i=1}^m (\NewC)^{K{_i}}(S_i) \rtarr (\NewC)^H(\amalg_i S_i)$$
which induce the desired monadic product $(\Cp_{\star}\Cp_{\star} \sX)(G/H) \rtarr (\Cp_{\star} \sX)(G/H)$.  We can use that $(\NewC)^H(T) =\sC(|T|)^{\GA_{\be}}$ to expand explicitly in terms of wreath products as defined in \autoref{wreath}.  Since this paper has no applications that need the expansion, we desist.  An essential point is that the input $\GA's$ of the structure maps determine the output $\GA's$ of their targets.

\subsection{Algebras over the monads of orbital presheaves} \label{sec:cpre-algebra}

An algebra over the monad $\Cp$ is an orbital presheaf $\sX \in \OG^{op}[\sT]$ together with a structure map $\Cp \sX \to \sX$ satisfying the
usual associativity and unit conditions.   The following result will give a key starting point for our applications.  It says that we are in the context of Assumption $A_{com}$ of \cite[Section 6.1]{KMZ1}.  Let $\om\colon \Cp \bR\rtarr \bR \bC$ denote the isomorphism of \autoref{prop:phi-C-commute}.

\begin{prop}\label{Acom} The following diagrams commute:
\[ \xymatrix{
& \bR  \ar[dl]_{\et \bR}  \ar[dr]^{\bR \et} & \\
\Cp \bR \ar[rr]_-{\om} &  & \bR\bC \\}
\ \ \ \ \ \ \
\xymatrix{
\Cp\Cp\bR \ar[r]^-{\Cp\om} \ar[d]_{\mu\bR} & \Cp\bR\bC \ar[r]^{\om \bC} & \bR\bC\bC \ar[d]^{\bR\mu}\\
\Cp\bR \ar[rr]_-{\om} & & \bR\bC\\}
\]
Therefore, for any $\bC$-algebra $X$, the orbital presheaf $\bR X$ is a $\Cp$-algebra.\
\end{prop}
\begin{proof} Modulo where it is put, $\et$ in the first diagram comes from the unit element, thought of as $\id\colon * \rtarr \sC(1)$, of the operad $\sC$ we start with.   Similarly, modulo placing it to the left of $\bS_{\star}$ or of $\bP\bQ$, the second diagram comes down to use of composition in the operad $\sC$, as dictated by \autoref{CtoCG}, and the construction of $\Cpre$.    A formal proof only elaborates this by explicitly writing down the maps.  Applying $\bR$ to the structure map $\Cmon X \to X$ of a $\bC$-algebra $X$, we obtain the structure map  $\Cp \bR X\iso \bR\bC X \rtarr \bR X$ of $\bR X$.
\end{proof}

The counit $\epz\colon \bL\bR \rtarr \Id$ is isomorphic to the identity natural transformation.  However, the following analog of the previous result is not as useful since the functor $\bL$ does not preserve weak equivalences, necessitating the intermediary of cofibrant approximation or the Elmendorf construction.
\begin{prop}
  \label{prop:theta-C}
 For any $\Cp$-algebra $\sX$, $\bL(\sX)$ is a $\Cmon$-algebra.
\end{prop}
\begin{proof}  Since the only subgroup of $G$ subconjugate to $e$ is $e$, inspection of the structure map $\Cp \sX \rtarr \sX$, evaluated on  $G/e$, shows that it is a structure map $\Cmon \bL \sX \rtarr \bL \sX$ of an action by $\Cmon$. 
\end{proof} 
  The two previous results have the following immediate implication.
\begin{cor}
  \label{cor:adjunction}
 The adjunction  $ \bL \dashv \bR$ restricts to an adjunction on algebras:
\begin{equation*}
  \begin{tikzcd}
    \bL: \Cp[ \OG^{op}[\sT]] \ar[r, shift left] & \Cmon[G\sT]: \bR \ar[l,
    shift left]
  \end{tikzcd}
\end{equation*}
\end{cor}

As noted in \autoref{humbug}, we also have the monad $\bR\bC\bL$ on orbital presheaves.  After application of cofibrant approximation, it can be used in place of $\Cp$ in the proof of \autoref{thm:pre-machine}.  However, in applications such as those of Sections \ref{units} and \ref{PICBR}, we start with an example of the form 
$\bR X$, then perform constructions on orbital presheaves that do {\em{not}} commute with $\bR$ to obtain the examples of interest.  Here application of 
$\bR\bC\bL$ does not work and we must use $\Cp$ instead. However, we do not need combinatorial precision in the description of $\Cp$ in such applications.


\section{Mackey functors and Eilenberg-MacLane $G$-spectra}\label{Examples}
 We consider the commutativity $G$-operad $\sN$.  Algebras over $\sN^{pre}$ give a new description of Mackey 
 functors and therefore a new construction of Eilenberg-MacLane $G$-spectra.
 
\subsection{The commutative case: $\bN^{pre}$-algebras}
\label{sec:mackey}
 Let $\sN$ be the commutativity $G$-operad with $\sN(n) = *$ for all $n$.  Using the description of $\Cp$ in \autoref{Cpre2}, we find  the following explicit description of $\mathbb{N}^{\mathrm{pre}}$ by classifying objects $T \in G\LA_*$ and breaking their groups $\SI_T$ into actions on orbits and permutations.

 \begin{exmp} \label{ex:Npre} Write $(H_i)$ for the $H$-conjugacy class of a subgroup $H_i$ of $H$.
\begin{align*} 
\mathbb{N}^{\mathrm{pre}}\sX(G/H) & = \big(\coprod_{T \in G\Lambda_{G/H}} \bS\sX(T )\big)/ \SI_T/\sim  \\
& =  \prod_{(H_i) \subset H} \mathbb{N} \big(\sX(G/H_i) / W_H H_i\big)
\end{align*}
For $K \subset H$, the restriction map $\phi^{*}: \mathbb{N}^{\mathrm{pre}}\sX(G/H) \rtarr
\mathbb{N}^{\mathrm{pre}}\sX(G/K)$ is induced by 
contravariant maps $\varphi^{*}: \sX(G/H_i) \rtarr \prod_j\sX(G/K_{ij})$ of the presheaf $\sX$, where $K_{ij} = H_i \cap
  g_jKg_j^{-1}$ for $g_j \in H_i\backslash G/K$ and
$\amalg_j G/K_{ij}$ is the pullback in the diagram 
\begin{equation}
  \label{eq:pull-back-Mackey}
  \begin{tikzcd}
    \amalg_j G/K_{ij} \ar[r] \ar[d,"\varphi"'] & G/K \ar[d,"\phi"]\\
    G/H_i \ar[r] & G/H.
  \end{tikzcd}
\end{equation}

The structure map $\mathbb{N}^{\mathrm{pre}} \mathbb{N}^{\mathrm{pre}} \sX(G/H) \rtarr 
\mathbb{N}^{\mathrm{pre}} \sX(G/H)$ is the map
\begin{equation}
  \label{eq:NpreComp0}
  \begin{tikzcd}
    \prod_{(H_i) \subset H} \mathbb{N}
    \bigg(\prod_{(H_{ij}) \subset (H_{i})}
    \mathbb{N} \Big(\sX(G/H_{ij}) / W_{H_i} H_{ij}\Big) / W_{H}H_{i} \bigg)
    \ar[d,"\mathbb{N}\prod \rtarr \prod\mathbb{N}"] \\
    \prod_{(H_{ij}) \subset (H_i)\subset H} \mathbb{N}
    \bigg( \mathbb{N} \Big(\sX(G/H_{ij}) / W_{H_i}H_{ij}\Big) / W_{H}H_{i}
    \bigg)
  \end{tikzcd}
\end{equation}
composed with the product over $(H_{ij}) \subset (H_i) \subset H$ of the following maps.
Here we abbreviate notation by writing $X = \sX(G/H_{ij})$.
\begin{equation}
  \label{eq:NpreComp}
  \begin{tikzcd}
  \mathbb{N} \bigg(\mathbb{N} (X / W_{H_i}H_{ij}) /W_{H}H_{i} \bigg)
    \ar[d,"\text{ induced by }"',"\bN (-)/W_H H_i \rtarr  \mathbb{N}(-/W_H H_i)"] \\
    \mathbb{N} \bigg(\mathbb{N} \big((X/W_{H_i}H_{ij})/ W_{H} H_{i} \big) \bigg)
    \ar[d,"\text{ induced by }"',"(X/W_{H_i}H_{ij})/ W_{H} H_{i} \rtarr X / W_{H}H_{ij}"]\\
   \mathbb{N} \bigg(\mathbb{N}\big(X / W_{H}H_{ij} \big) \bigg)
    \ar[d,"\mathbb{N} \circ \mathbb{N} \rtarr \mathbb{N}"] \\
    \bN(X/ W_{H}H_{ij}) 
  \end{tikzcd}
\end{equation}
and then ignoring the now superfluous index $(H_i)$.  Here $X$ has an action by $W_HH_{ij}$ and $W_{H_i}H_{ij} \subset W_HH_{ij}$; the first arrow is induced by the resulting quotient  map.
\end{exmp}

Algebras over $\mathbb{N}^{\mathrm{pre}}$ can be explicitly described as follows.
\begin{prop} \label{prop:alg-N}
  An $\mathbb{N}^{\mathrm{pre}}$-algebra $\sX$ consists of commutative topological monoids
  $\sX(G/H)$ for $H \subset G$, isomorphisms of monoids $c_g: \sX(G/H) \rtarr
  \sX(G/gHg^{-1})$, and $W_H K$-equivariant maps $t_K^H: \sX(G/K) \rtarr \sX(G/H)$ for
  $K \subset H$ (where $W_H K$ acts trivially on $\sX(G/H)$), such that
\begin{enumerate}[(i)]
\item \label{item:alg-N-1}The map $t_K^H$ is a map of commutative monoids.
\item \label{item:alg-N-plus} $c_g\circ c_h = c_{gh}$ and $c_g \circ t_K^H =
  t_{gKg^{-1}}^{gHg^{-1}} \circ c_g$.
\item \label{item:alg-N-2}For $K,H_i \subset H$ and $K_{ij} \subset K$  as in
  \autoref{eq:pull-back-Mackey},
 the diagram below (whose vertical maps are contravariant
  maps of the presheaf) is commutative
\begin{equation}\label{eq:compatibility}
  \begin{tikzcd}
    \sX(G/H_i) \ar[rr,"t_{H_i}^H"] \ar[d,"\varphi^{*}"'] & & \sX(G/H) \ar[d,"\phi^{*}"] \\
    \prod_i\sX(G/K_{ij}) \ar[r,"\prod_it_{K_{ij}}^K"'] & \prod_i\sX(G/K) \ar[r,"\text{mult}"'] & \sX(G/K)
  \end{tikzcd}
\end{equation}
\item \label{item:alg-N-3} For $L \subset K \subset H$, $t_K^H \circ t_L^K = t_L^H$, and similarly for inclusions of subconjugates.
\end{enumerate}
\end{prop}
\begin{proof}
First suppose that we have commutative monoid structures on the $\sX(G/H)$ 
and $W_HK$-maps $t_K^H$
such that (\ref{item:alg-N-1}) through (\ref{item:alg-N-3}) are
satisfied. For $H_i \subset H$, the map $t_{H_i}^H$ determines a map
\begin{equation}
\label{eq:comm1}
\xymatrix@1{T_{H_i}^{H}: \mathbb{N}(\sX(G/H_i)/{W_H H_i}) \ar[r]^-{\mathbb{N}(t_{H_i}^{H})} & \mathbb{N}(\sX(G/H)) \ar[r]^-{\mathrm{mult}} & \sX(G/H).\\}
\end{equation}
By \autoref{ex:Npre}, these maps determine a map
\begin{equation}\label{eq:comm2}
\xymatrix@1{\mathbb{N}^{\mathrm{pre}} \sX(G/H)  \ar[rr]^-{\prod_i\mathbb{N}(T_{H_i}^H)} & &  \prod_{(H_i) \subset H}\mathbb{N}\sX(G/H) \ar[r]^-{\mathrm{mult}}& \sX(G/H).\\}
\end{equation}
For $K \subset H$, we also have the map
\begin{equation*}
\xymatrix@1{ \iota_{K}^{H}: \sX(G/K) \ar[rr]^-{\text{ quotient }} & & \sX(G/K)/{W_H(K)} \ar[rr]^-{\text{inclusion}}  & & \mathbb{N}(\sX(G/K)/{W_H(K)}),\\}
\end{equation*}
and the unit map $\sX \rtarr \mathbb{N}^{\mathrm{pre}}\sX$ is induced by $\iota_H^H$.
It suffices to check that these maps give an $\bN^{pre}$-algebra structure map.  This means that they satisfy the following
conditions.  We omit the easy verifications.  
\begin{enumerate}
\item \label{item:Nalg1}(unital) The following diagram is commutative.
 $$\xymatrix{
\sX(G/H) \ar[r]^-{ \iota_H^H} \ar@{=}[dr]& \mathbb{N}(\sX(G/H)) \ar[d]^{T_H^H} \\
& \sX(G/H).\\}$$
\item \label{item:Nalg2}(associative) For $H_{ij} \subset H_i \subset H \subset G$, the following diagram is commutative.
\begin{equation*}
  \begin{tikzcd}
      \mathbb{N} \bigg(
    \mathbb{N} \big(\sX(G/H_{ij}) / W_{H_i}(H_{ij})\big) / W_{H}(H_{i}) \bigg)
    \ar[r,"\autoref{eq:NpreComp}"] \ar[d, " \mathbb{N}(T_{H_{ij}}^{H_i})"']&
    \mathbb{N} \big(\sX(G/H_{ij}) / W_{H}(H_{ij}) \big)  \ar[d,"T_{H_{ij}}^H"]
    \\
    \mathbb{N} \bigg(\sX(G/H_{i}) / W_{H}(H_{i}) \bigg)
    \ar[r,"T_{H_i}^H"'] &
    \sX(G/H)
  \end{tikzcd}
\end{equation*}
\item \label{item:Nalg3} (map of presheaves) For $H, K, K_{ij}$ as in 
  \autoref{eq:pull-back-Mackey}, the
  following diagram (where the vertical maps are contravariant maps of the
  presheaf) is commutative.
\begin{equation*}
  \begin{tikzcd}
    \mathbb{N} \bigg(\sX(G/{H_i}) / W_{H}({H_i}) \bigg)
    \ar[rr,"T_{H_i}^H"] \ar[d," \mathbb{N}(\varphi^{*})"'] & & \sX(G/H) \ar[d,"\phi^{*}"] \\
    \prod_i\mathbb{N} \bigg(\sX(G/K_{ij}) / W_{K}(K_{ij}) \bigg) \ar[r,"\prod_iT_{K_{ij}}^K"'] & \prod_i\sX(G/K) \ar[r,"\text{mult}"'] & \sX(G/K)  
  \end{tikzcd}
\end{equation*}
\end{enumerate}
For the converse, suppose that $\sX$ is an $\mathbb{N}^{\mathrm{pre}}$-algebra.  With the notations at the start of \autoref{ex:Npre}, we then have structure maps
\begin{equation}
\label{eq:comm-T}
T_{H_i}^{H}: \mathbb{N}(\sX(G/H_i)/{W_H(H_i)})\rtarr \sX(G/H)
\end{equation}
for each $H_i\subset H$. Taking $H_i=H$, this gives $\sX(G/H)$ its commutative monoid structure.  
Define the $W_HK$-map $t_K^H\colon \sX(G/K) \rtarr \sX(G/H)$ to be the composite $T_K^H \circ \iota_K^H$.
We must verify statements (\ref{item:alg-N-1}) through (\ref{item:alg-N-3}), and we have
the statements  (\ref{item:Nalg1}), (\ref{item:Nalg2}), and (\ref{item:Nalg3}) by hypothesis.
Taking $H_i = H$ in (\ref{item:Nalg2}) and precomposing the diagram on
top by $\mathbb{N}(\iota_{H_{ij}}^H)$ and $\iota_{H_{ij}}^H$, we see that
$t_{H_{ij}}^H$ is a map of commutative monoids. This proves
(\ref{item:alg-N-1}).   Taking $\phi$ to be $G/gHg^{-1} \to G/H$ in (\ref{item:Nalg3}), we have $c_g T_K^H = T_{gKg^{-1}}^{gHg^{-1}} c_g$. Naturality gives $c_g \iota_K^H = \iota_{gKg^{-1}}^{gHg^{-1}} c_g$. This proves (\ref{item:alg-N-plus}).
Precomposing the diagram in (\ref{item:Nalg3}) on the left
with $\iota_{H_i}^H$ and $\iota_{K_{ij}}^K$ proves (\ref{item:alg-N-2}). Precomposing
the diagram in (\ref{item:Nalg2}) with $\iota$ on top and left proves (\ref{item:alg-N-3}).
 \end{proof}
 
 \subsection{Mackey functors}
 Mackey functors have many equivalent definitions.  We take the definition that a Mackey functor  is a covariant transfer functor and a contravariant restriction functor (with the same values on objects)  from the category of finite $G$-sets to the category of abelian groups that is additive (takes disjoint unions to direct sums) and takes pushouts to appropriate commutative diagrams (e.g. \cite[p. 209]{EHCT}).
 
 Adding disjoint basepoints to finite $G$-sets, we see that if we have a Mackey functor $M$, its restriction maps 
give an orbital presheaf $\sX_M$, and its transfer maps satisfy 
\autoref{prop:alg-N}(\ref{item:alg-N-1})-(\ref{item:alg-N-3}).  Since the $M(G/H)$ are abelian groups, not just monoids, $\sX_M$ is grouplike.  Conversely, if $\sX$ is a discrete grouplike $\bN^{pre}$-algebra, so that each $\sX(T)$ is a set viewed as a space and is an abelian group, then it gives a Mackey functor via these conditions.  The essential point is that passage from the pullback diagram 
\autoref{eq:pull-back-Mackey} to the commutative diagram \autoref{eq:compatibility} is just a disguised version of the double coset formula. 
This gives a new characterization of Mackey functors.

 \begin{thm}\label{HM}
  \label{cor:MackeyN} The category of Mackey functors is isomorphic to the category of discrete grouplike $\mathbb{N}^{\mathrm{pre}}$-algebras.
  \end{thm}
  
Since $\sN$ is the terminal operad, for any group $G$ and any operad $\sC$ of $G$-spaces, there is a map of operads $\sC \rtarr \sN$.  It induces a map of monads $\Cp \rtarr \mathbb{N}^{\mathrm{pre}}$. Pullback along this map gives a forgetful functor 
$$\mathbb{N}^{\mathrm{pre}}\big[\OG^{op}[\sT]\big] \rtarr \Cp\big[\OG^{op}[\sT]\big].$$
Taking $\sC$ to be an $E_{\infty}$ operad, an immediate application of the functor $\mathbb{E}^{\mathrm{pre}}$ from \autoref{thm:pre-machine} gives us Eilenberg-MacLane $G$-spectra.

\begin{thm}\label{HM1} Let $M$ be a Mackey functor. Then 
  $$\ul\pi_{n}\mathbb{E}^{\mathrm{pre}}M = \begin{cases}
M & n= 0 \\
0 &  n\neq 0
\end{cases}$$
Thus $\bE^{\mathrm{pre}}M$ is an Eilenberg-Mac\,Lane $G$-spectrum $HM$.  
\end{thm}
\begin{proof}
By \autoref{thm:pre-machine}, the discreteness and grouplike properties give
$$ M= \ul{\pi}_0(M) = \ul{\pi}_{*}(M) \cong \ul{\pi}_*(\Omega^{\infty}\mathbb{E}^{\mathrm{pre}}M) \cong \ul{\pi}_*(\mathbb{E}^{\mathrm{pre}}M). \qedhere $$
\end{proof}

\begin{rem} A multiplicative refinement is work in progress.  We intend to prove that $\bE^{pre}M$ is an $E_{\infty}$ ring $G$-spectrum when $M$ has appropriate multiplicative structure.  We hope to obtain  
new interpretations of Green and Tambara functors.
\end{rem}

Speculating on their future interest, we record the following definition and theorem, which delete the discreteness condition.  They give the equivariant analog of topological abelian groups together with their associated equivariant spectra.

\begin{defn}  A topological Mackey functor is a grouplike $\bN^{pre}$-algebra.
\end{defn}

\begin{thm}\label{HMTop} $\bE^{pre}$ specializes to give a functor from topological Mackey functors to $G$-spectra such that $\ul{\pi}_*(\bE^{pre}M)$ is isomorphic to $\ul{\pi}_*(M)$.
\end{thm} 
\begin{proof}
Since $M \rtarr \OM^{\infty}\bE^{pre}M$ induces a levelwise homology isomorphism of grouplike commutative topological monoids, it is a levelwise weak equivalence.
\end{proof}

\section{Unit $G$-spectra of $E_{\infty}$ ring $G$-spectra}\label{units}
In this section, we apply the presheaf infinite loop space machine to construct the
unit spectra of equivariant commutative ring spectra, alias $E_{\infty}$ ring $G$-spectra.

\subsection{Nonequivariant unit spectra}\label{unitsp}
We again use Lewis-May spectra, so that they have well-structured zeroth spaces.
We first recall some non-equivariant definitions.
For a ring spectrum $R$, its zeroth space $R_0$ inherits a ring structure and is
therefore a ring space.
Define $GL_1(R)$ to be the subspace of $R_0$ consisting of the unit components, namely the pullback  
$$
\begin{tikzcd}
GL_1(R)\ar[r]\ar[d] \arrow[dr, phantom, "\lrcorner"', very near start] & R_0\ar[d] \\	
\pi_0(R)^\times\ar[r] & \pi_0(R),
\end{tikzcd}
$$
where $(-)^\times$ denotes the subset of units in the ring $\pi_0(R)$.
When $R$ is an $\mathrm{E}_{\infty}$-ring spectrum, $R_0$ is an $\mathrm{E}_\infty$-ring
space, so that $GL_1(R)$ is also an $\mathrm{E}_{\infty}$-space. This is so since the
multiplicative operad action on $R_0$ restricts to give an equivariant $E_\infty$-structure on
$GL_1(R)$. Consequently one can apply the infinite loop space machine to obtain
a spectrum $gl_1(R)$, the \emph{unit spectrum} of $R$, such
that $(gl_1(R))_0\simeq GL_1(R)$.  The original source is \cite[Section IV.3]{MQR} and a more recent source is \cite[Section 2]{Rant2}, both focused on the role of $GL_1R$ in orientation theory. 

\subsection{Unit $G$-spectra}\label{Gunitsp}
Equivariantly, it is not straightforward to define a single $G$-space of units of an $E_{\infty}$ ring $G$-spectrum $R$ whose fixed points capture the equivariant units. We shall use fixed point orbital presheaves.

\begin{defn}
  Let $R$ be an $E_{\infty}$ ring $G$-spectrum.
  Define the unit fixed point presheaf $\ul{GL}_1(R): \OG^{op} \rtarr \sT$ by
  $$\ul{GL}_1(R)(G/H) = {GL_1}((R_0)^H).$$ 
  This is well-defined because the
  restrictions in the Green functor $\ul{\pi}_0(R)$ are ring maps.
  Indeed, we have a pullback diagram of presheaves
\begin{equation}
    \label{eq:defn-GL1}
\begin{tikzcd}
\ul{GL}_1(R)\ar[r]\ar[d] \arrow[dr, phantom, "\lrcorner"', very near start]& \fixpt(R_0)\ar[d] \\	
\ul{\pi}_0(R)^\times\ar[r] & \ul{\pi}_0(R)
\end{tikzcd}
\end{equation}
\end{defn}

Let $\mathscr{C}$ be an $\mathrm{E}_{\infty}$ $G$-operad. The most common choice in this context is 
the linear isometries operad associated to the complete universe $U$, as in \cite{LMS, EKMM}. 
 We have the presheaf monad
$$\Cp: \OG^{op}[\sT] \rtarr \OG^{op}[\sT].$$
A $\sC$-algebra in $G$-spectra is an $\mathrm{E}_{\infty}$ ring
$G$-spectrum. Let $R$ be such a $G$-equivariant $\sC$-spectrum, so that $R_0$ is a
$\sC$-algebra in $G\sT$. 
\begin{thm}
\label{thm:units_spec_LU-algebra}
$\ul{GL}_1(R)$  is a $\Cp$-algebra.   
\end{thm}
\begin{proof}
By \autoref{Acom},  $\fixpt(R_0)$  is a $\mathbb{C}^{pre}$-algebra.
 By \autoref{cor:MackeyN}, $\ul{\pi}_0(R)$ and $\ul{\pi}^{\times}_0(R)$ are
 $\mathbb{N}^{\mathrm{pre}}$-algebras and therefore  $\bC^{pre}$-algebras.
 It follows that  \autoref{eq:defn-GL1} is a pullback  square of $\bC^{pre}$-algebras. 
\end{proof}


\begin{prop}
  \label{prop:phi-C-alge}
The $G$-space $GL_1(R_0) = \ul{GL}_1(R)(G/e)$ is a $\sC$-space.
\end{prop}
\begin{proof}
  This is immediate from \autoref{prop:theta-C} and \autoref{thm:units_spec_LU-algebra}. 
  \end{proof}
  
A direct alternative proof may be illuminating. The essential point is that the restriction to  $\sC(n) \times GL_1(R_0)^n$ of the structure map  
 $\sC(n) \times (R_0)^n \rtarr R_0$ factors through the pullback $GL_1(R_0)$. This holds if 
$$
\xymatrix@1{\pi_0\Big(\sC(n) \times GL_1(R_0)^n\Big) \ar[r]^-{\subset} & \pi_0\Big(\sC(n) \times (R_0)^n\Big) \ar[r] & \pi_0(R_0)\\}
$$
factors through $\pi_0(R_0)^{\times}$. Since $\sC$ is an $E_\infty$ operad, $\pi_0(\sC(n)) = *$.  Therefore this holds if and only if 
$$
\xymatrix@1{\Big(\pi_0(R_0)^{\times}\Big)^n \ar[r]^-{\subset} & \Big(\pi_0(R_0)\Big)^n \ar[r]^-{mult} & \pi_0\Big(R_0\Big),\\}
$$
factors through $\pi_0(R_0)^{\times}$, and it does since multiplication in a ring preserves units.

The unit of the adjunction $(\bL,\bR)$ gives a map
$$\eta: \ul{GL}_1(R_0) \rtarr \fixpt \underlying \ul{GL}_1(R) = \fixpt GL_1(R).$$
At $G/H$, $\et$ is a map
\begin{equation}
\label{eq:unit_of_adjunction}
	\eta(G/H):GL_1(R_0^H)\rtarr GL_1(R_0)^H.
      \end{equation}
\begin{exmp} \label{rem:not_equivalence}	
We show by example that the map \autoref{eq:unit_of_adjunction} is in general {\em{not}}  an
 equivalence.  Therefore the underlying $G$-space $GL_1(R_0)$ is in general not the equivariant
 unit space for the ring $G$-spectrum $R$.   We have a commutative diagram
   \begin{equation*}
  \begin{tikzcd}
   GL_1(R_0^H) \ar[r] \ar[d] & R_0^H \ar[d] \\
    GL_1(R_0) \ar[r] & R_0
  \end{tikzcd}
\end{equation*}
Since $GL_1(R_0)^H$ is the pullback  of the bottom and right arrows, 
the map  \autoref{eq:unit_of_adjunction} is an equivalence if and
 only if the diagram is a homotopy pullback.
This is true if and only if the next commutative diagram is a pullback of sets:
\begin{equation*}
  \begin{tikzcd}
    \pi_0^H(R)^{\times} \ar[r] \ar[d] & \pi_0^H(R) \ar[d] \\
    \pi_0(R)^{\times} \ar[r] & \pi_0(R)
  \end{tikzcd}
\end{equation*}
Let $H = G=C_3$, the cyclic group of order $3$, and let $R$ be the sphere spectrum $\mathbb{S}_{G}$.  
Then $\pi_0^{G}( \mathbb{S}_{G})$ is the Burnside ring of $G$, which in this case is  $\mathbb{Z}[x]/(x^2 =3x)$.  Its unit group
is $\pi_0^{G}( \mathbb{S}_{G})^{\times} = \{\pm 1\}$.  But the pullback of the diagram is $\{a+bx|a+3b=\pm 1\}$.
\end{exmp}

\begin{thm}\label{unit}
	Let $R$ be an $\mathrm{E}_{\infty}$-ring $G$-spectrum. Then there is a $G$-spectrum $gl_1(R)$ such that 
	$\fixpt(\OM^\infty gl_1(R))$ is equivalent to $\ul{GL}_1(R)$.
\end{thm}
\begin{proof} By assumption, $R$ is an algebra over an
 $\mathrm{E}_{\infty}$ $G$-operad $\sC$.
 By \autoref{thm:units_spec_LU-algebra}, 
  $\ul{GL}_1(R)$ is a $\Cp$-algebra. It is grouplike since
$\pi_0 (\ul{GL}_1(R)(G/H)) \cong \pi^H_0(R)^{\times}$
is an abelian group for $H \subset G$.
Applying \autoref{thm:pre-machine}, let 
\begin{equation*}
gl_1(R) = \mathbb{E}^{pre}(\ul{GL}_1(R)).
\end{equation*}
Then $  \ul{GL}_1(R) \simeq \fixpt(\OM^\infty gl_1(R))$, as desired.
\end{proof}

\section{Picard and Brauer $G$-spectra of $E_{\infty}$ ring $G$-spectra}\label{PICBR}
\subsection{Permutative $G$-categories and algebraic $K$-theory $G$-spectra}\label{Perm}

Our construction of Picard and Brauer $G$-spectra is analogous to the construction of equivariant algebraic $K$-theory via infinite loop space theory that is given in \cite[Section 4.4]{GM3}.  We review that here, since a slightly new perspective on how  the difference between classical (alias naive) and genuine $G$-spectra plays out in equivariant infinite loop space theory will be a focal point of our discussion.  

Let $\sE S$ denote the indiscrete (alias chaotic) category of a set of objects $S$; there is a unique morphism 
$s\rtarr t$ for $s,t\in S$.  Being old-fashioned, we write $B$ for the classifying space functor, the realization of the nerve, from categories to spaces.   Then $B\sE S$ is contractible for any $S$.   We write $EG =  B\sE(G)$. It is a contractible space with a free (right) action by $G$, and it is the universal cover of its orbit space $BG = EG/G$.  Of course, as we shall occasionally use, the construction generalizes directly to give the principal $G$-bundle $EG\rtarr BG$ of a topological group $G$. In what follows, we mostly revert to our assumption that $G$ is finite, but we shall sometimes use that our constructions work the same way more generally.

The permutativity operad $\sP$ in $\sC at$ has $\sP(j)=E\sM(j)=E\SI_j$, where $\sM$ is the associativity operad. For now, we write  $\sP^{top}$ for the operad in spaces obtained by applying $B$ to $\sP$. (We will sometimes omit the superscript `top' later.)   It is an $E_{\infty}$ operad in the category $\sU$ of (unbased) spaces.  This is the Barratt-Eccles operad.   We regard spaces as $G$-trivial $G$-spaces. As noted in \cite[Remark 5.62]{KMZ1}, we have a classical analog of \autoref{machine}, giving an infinite loop space machine from  $\sP^{top}$-$G$-spaces (possibly $G$-trivial and possibly topological) to classical $G$-spectra.  Analogously, we have a classical analog of \autoref{thm:pre-machine}, giving an infinite loop space machine from $\bP^{pre}$-orbital presheaves to classical $G$-spectra.

Returning to categories, a $\sP$-$G$-category (again possibly $G$-trivial and possibly topological) $\sA$ is a classical (alias naive),  permutative $G$-category.  As in \cite[Definition 4.4]{GM3}, define $\sP_G$ to be the operad $\sC at(\sE G, \sP)$. By definition, a {\em {genuine}}  permutative $G$-category is a $\sP_G$-$G$-category, and that is what we understand to be the right meaning of a permutative $G$-category when $G$ is finite.  

\begin{defn}\label{classtogen}  For a classical permutative $G$-category $\sA$, let $\sA_G = \sC at(\sE G, \sA)$.
As in \cite[Proposition 4.6]{GM3}, $\sA_G$ is a genuine permutative $G$-category, and these are the canonical   examples. Define $\io \colon \sP\rtarr \sP_G$ and $\io \colon \sA \rtarr\sA_G$ to be the functors induced by the projection from $\sE G$ to the trivial $G$-category $*$.  The first is a map of operads that induces a forgetful functor $\io^*$ from $\sP_G$-$G$-categories to $\sP$-$G$-categories, and the second is a map of $\sP$-categories
$\io \colon \sA \rtarr \io^*\sA_G$.
\end{defn}

Observe that $(B\sA_G)^G$ is a model for the homotopy fixed point space of the $G$-space $B\sA$. 
Write $\sP_G^{top}$ for the operad  in $G$-spaces obtained by applying $B$ to $\sP_G$. It is an $E_{\infty}$ operad in the category $G\sU$ of (unbased)  $G$-spaces.  The classifying space of a $\sP_G$-$G$-category $\sB$ is a $\sP_G$-$G$-space. Via \autoref{machine}, it has an associated $G$-spectrum as in \cite[Section 5.6]{KMZ1}.

\begin{defn}\label{Kthy1A}  The classical $K$-theory $\bK\sA$ of a $\sP$-$G$-category $\sA$ is the classical $G$-spectrum of the $\sP^{top}$-space $B\sA$.   Its $0$th space is a group completion of $B\sA$.  The (genuine) $K$-theory  $\bK_G\sB$ of a $\sP_G$-$G$-category  $\sB$ is the (genuine) $G$-spectrum  of the $\sP_G^{top}$-$G$-space $B\sB$.   Its zeroth $G$-space is a group completion of $B\sB$.  In particular, we have the $K$-theory $G$-spectrum $\bK_G\sA_G$ of a  $\sP$-$G$-category $\sA$.
\end{defn}

\begin{rem} As in \cite[Remark 4.11]{GM3}, we define a symmetric monoidal $G$-category to be a pseudoalgebra over $\sP_G$.   A classical symmetric monoidal category is a pseudoalgebra over $\sP$. It is classical category theory that classical symmetric monoidal categories can be strictified to classical permutative categories.  The genuine analog is  proven in \cite{GMMO1}.   Thus  there is no loss of generality in restricting attention to permutative $G$-categories of either type.
\end{rem}

By \cite[Theorem 4.15]{GM3}, for $\sC_G$-$G$-spaces $\sA$ and $\sB$, there is a natural weak equivalence
$$\bK_G(\sA\times \sB) \rtarr \bK_G\sA \times \bK_G\sB$$ 
of genuine $G$-spectra. The analogous proof shows that for $\sC$-$G$-spaces $\sA$ and $\sB$, there is a natural
weak equivalence
$$\bK(\sA\times \sB) \rtarr \bK\sA \times \bK\sB$$ 
There are other formal similarities between $\bK$ and $\bK_G$.

\subsection{The homotopy limit problem}\label{lim}

However, we are interested in the comparison between $\bK$ and $\bK_G$.  Understanding this boils down to understanding an example of the homotopy limit problem, as articulated by Thomason  \cite{Thom}; see also Merling \cite{Mer}.  

Write $i^*$  for the forgetful functor from genuine $G$-spectra to classical $G$-spectra. Then \cite[Theorems 2.24 and 4.16]{GM3} show that, for a $\sP_G$-$G$-category $\sB$, such as $\sA_G$ for a classical permutative category 
$\sA$, there is a natural weak equivalence
 \begin{equation}\label{Compone} 
  \bK\io^*\sB \rtarr i^*\bK_G\sB
 \end{equation}
of classical $G$-spectra.   Here $\io^*\sB$ denotes $\sB$ regarded as $\sP$-category by restriction along $\io\colon  \sP\rtarr \sP_G$.  Remembering that $\sA_G = \sC at(\sE G, \sA)$,  we have the composite
\begin{equation}\label{Yeah2}
\xymatrix@1{
\bK \sA \ar[r]^-{\bK \io}  & \bK\io^*\sA_G \ar[r]^-{\htp} &  i^*\bK_G\sA_G\\} 
\end{equation}

The map $\bK{\io}$  is sometimes an equivalence \cite[Proposition 4.7 and Example 4.19]{GM3} and sometimes not \cite[Proposition 4.19]{GMM}.  We introduce the following name.                         

\begin{defn}\label{amenable} We say that a classical permutative $G$-category $\sA$ (possibly $G$-trivial) is 
{\em {amenable}}  if the functor $\io\colon \sA \rtarr \io^*\sC at(\sE G, \sA)$ induces a weak equivalence on passage to classifying $G$-spaces.  The induced map $\bK\io$ of \autoref{Yeah2}  is then a weak equivalence.
\end{defn}

The problem of determining whether or not $\sA$ is amenable can be viewed as an example of the homotopy limit problem formulated by Thomason \cite{Thom}.
Following Street \cite{Street}, his \cite[Section 2]{Thom} identifies $\sC at_G(\sE G,\sA)$ as the ``$laxlim_G(\sA)$" and gives an explicit description. The subscript $G$ means that he is looking nonequivariantly at the $G$-equivariant functors.  The limit here is just passage to $G$-fixed points.   As he notes in \cite[Section 3]{Thom}, but generalizing by allowing a non-trivial action of $G$ on $\sA$,  after applying $B$ and an infinite loop space machine, the functor $\io^*$ induces a map of classical $G$-spectra.  The homotopy limit problem asks how close this map is to being a weak equivalence.   

For a $G$-spectrum $E$ and a $G$-space (or spectrum) $X$,  the map  $EG_+ \rtarr  S^0$ induces a map $E^*(X)\rtarr E^*(EG_+\sma X)$, and one can ask what this map does.  The  Atiyah-Segal completion theorem, the Segal conjecture, and other such results  are examples where this question is answered in terms of suitable completions.  When $X=S^0$, this is a topological variant of Thomason's categorical problem. Thomason shows that the Quillen-Lichtenbaum conjecture is another example of the problem. 

\begin{sch}\label{Lenz} The abstract of  \cite{Lenz} says ``We prove that through the eyes of equivariant {\em weak} equivalences the genuine symmetric monoidal G-categories of \cite{GMMO1} are equivalent to just ordinary symmetric monoidal categories with G-action."  However, \cite[Proposition 4.7]{GM3}  implies that, with the definition of equivariant {\em weak} equivalences given in  \cite[Definition 7.1]{Lenz}, $\io \colon \sA \rtarr \sC at(\sE G,\sA)$ is itself a {\em weak} equivalence for all  $G$-categories $\sA$.  Thus his definition of {\em weak} equivalence is so weak that it is blind to the homotopy limit problem. 
\end{sch} 

We will only be interested in the case when $G$ acts trivially on the set of objects of $\sA$, meaning that each object $A$ is itself a $G$-object.  In that case, 
$$\sC at_G(\sE G,\sA) = \sC at(\sE G,\sA)^G$$ 
is the $G$-fixed point category obtained by passing from morphism $G$-functors to morphism functors. We can restrict from $G$ to $H\subset G$ before passing to fixed points, and then the problem asks more generally whether or not passage to fixed points results in a weak equivalence of $H$-fixed point spectra for all $H\subset G$.  Adding in maps between orbits, the amenability question asks for which classical permutative $G$-categories $\sA$ the maps 
$$\io\colon \sA^H \rtarr \io^*\sC at(\sE G,\sA)^H = \io^*\sA_G^H$$
induce a weak equivalence of orbital presheaves  upon application of the classifying space functor $B$.  On the left, to be pedantic, we should write $\sA|_{H}$ for the implicit restriction of the $G$-category $\sA$ to an $H$-category. 

\begin{rem} At this writing, we have not thought seriously about when amenability holds.  The idea is that any $\sP$-$G$-category $\sA$ gives rise to a $\sP_G$-$G$-category $\sA_G$ and we ask how far they are from giving equivalent classical $G$-spectra.  There is a converse question.  Start with any genuine $\sP_G$-algebra $\sB$ and regard it as a $\sP$-algebra $\io^*\sB$.  Is it amenable?  We think the answer should be yes.  As noted above,  \cite[Proposition 4.7]{GM3} shows that this is true when $\sB = \sC at_G(\sE G,\sA)$ for some $\sP$-algebra $\sA$.
\end{rem}

For a $G$-ring $R$ (possibly $G$-trivial), the disjoint union of the $GL(n,R)$ is a classical permutative $G$-category $\sG\sL(R)$ under block sum of matrices.   We abbreviate notation by writing
\begin{equation}\label{KR}
 \bK(R) = \bK(\sG\sL(R)) \ \ \text{and} \ \  \bK_G(R) = \bK_G(\sG\sL(R)_G).
\end{equation}
The first is a classical $G$-spectrum whose $0$th $G$-space is a group completion of $B\sG\sL(R)$, and the second is a genuine $G$-spectrum whose $0$th $G$-space is a group completion of $B\sG\sL(R)_G$. 

In particular $R$ might be a $G$-field $E$, where $E$ is a Galois extension of a field $F$ with Galois group $G$.  Then \cite[Example 4.20]{GMM} (see also \cite[Example 4.19]{GM3}) gives that passage to fixed point categories 
gives an equivalence 
$$ GL(n, E^H) \rtarr \sC at(\sE G, GL(n,E))^H$$
for each $H\subset G$.
In fact, \cite{GMM} explains that this is a direct consequence of Serre's general version of Hilbert's Theorem 90    \cite[Ch. 10, Prop. 3]{Serre} and \cite[Section 6.2]{Mer} reproves Serre's result in our context.   Therefore  
$\sG\sL(E)$ is amenable.  This gives the desired conclusion that the natural map
$$ \bK(E^H) \rtarr  (\bK_G E)^H$$
is a weak equivalence for all $H\subset G$.    

It is natural to extend the theory from $G$-rings $R$ to $G$-ring spectra $R$.  We say more about what we mean by those in \autoref{GPicR}.  We define $GL(n,R)$ to be the topological group of automorphisms of the $R$-module $R^n$.   The disjoint union over $n$ of these topological categories gives a permutative $G$-category $\sG\sL (R)$.  With $\sA = \sG\sL (R)$, the discussion above generalizes directly.    This raises the following natural question.

 \begin{quest}[Spectral Hilbert 90]\label{Spec90} Let $R$ be a commutative ring $G$-spectrum.   Defining $\bK(R)$ to be the classical $G$-spectrum obtained from the $\sP$-$G$-category $\sG\sL(R)$ and $\bK_G(R)$ to be the genuine $G$-spectrum obtained from the $\sP_G$-$G$-category  $\sC at(\sE G, \sG\sL(R))$, what conditions ensure that $\sG\sL(R)$ is amenable, so that the natural map
 $$\bK(R^H) \rtarr  (\bK_G R)^H$$
 is a weak equivalence for all $H\subset G$?
 \end{quest}

\subsection{Picard categories and nonequivariant Picard spectra}\label{PicR}

The results of this section were long common knowledge, but we mostly follow the exposition in Szymik \cite{Szymik}. 
Let $R$ be a  commutative $S$-algebra, in the sense of \cite{EKMM}. That is, $R$ is a commutative monoid in the symmetric monoidal category of spectra defined there.   (We can assume that $R$ is cofibrant, and  all spectra are fibrant in the context of \cite{EKMM}.)   We have the groupoid $\sP ic(R)$ of invertible $R$-modules and isomorphisms between them.  Its object set consists of one cofibrant choice from each weak equivalence class of invertible $R$-modules.  Its space of maps $M\rtarr N$ is the space of isomorphisms of $R$-modules. We denote its classifying space (the realization of its nerve) as ${\bf{Pic}}(R)$.  Thus $\pi_0 {\bf{Pic}}(R)$ is the Picard group $Pic(R)$ of equivalence classes of invertible $R$-modules, which is an abelian group under smash product over $R$.   

Since $\sP ic(R)$ is symmetric monoidal under the smash product,  applying the Segal or (replacing our groupoid by an equivalent permutative groupoid) the operadic infinite loop space machine,  we  obtain a spectrum ${\bf{pic}}(R)$ whose zeroth space $\OM^{\infty}{\bf{pic}}(R)$ is equivalent to  ${\bf{Pic}}(R)$.  Since all components of the grouplike 
$E_{\infty}$ space ${\bf{Pic}}(R)$ are equivalent and the component of the basepoint $R$ is the group $GL_1(R)$ of self-equivalences of $R$, we have
$$  {\bf{Pic}}(R) \htp Pic(R) \times BGL_1(R). $$
Therefore the loop space of ${\bf{Pic}}(R)$ is $GL_1(R)$, giving this space a possibly different infinite loop space structure from the one considered in \autoref{Gunitsp} with $G=e$.

When  $R$ is the sphere spectrum (or any other $E_{\infty}$ ring spectrum that comes in the same way from an ``$\sI$-space"), these two infinite loop structures are shown to be equivalent in \cite{MayIMon}.  The ideas in that paper can be modernized to show that this remains true for any commutative $S$-algebra $R$, but the details would take us too far afield.  We assume that is true, or we replace the infinite loop space of the previous section with the new one just obtained. Since $\OM {\bf{Pic}}(R) =GL_1(R)$,  $\pi_{n+1}{\bf{Pic}}(R) \iso \pi_n GL_1(R)$ if $n\geq 0$.  Note that $\pi_n(GL_1(R)) = \pi_n(R)$ if $n\geq 1$.  On the spectrum level,  $\SI gl_1(R)$ is a connective cover of ${\mathbf{pic}}(R)$. 

\subsection{Equivariant Picard spectra}\label{GPicR}  We would like an analogous picture equivariantly.  As explained in \cite[Chapter XXIV]{EHCT} and elaborated in \cite{MM},  although it is written non-equivariantly, the theory of \cite{EKMM} applies nearly verbatim equivariantly.   Thus let $R$ be a commutative $S_G$-algebra, that is a commutative monoid in the symmetric monoidal category of $G$-spectra.   

Form the topological $G$-category $\sP ic_G(R)$ of invertible $R$-modules.  The subscript $G$ just indicates that we are working in the category of $G$-spectra.
 Its objects are one chosen (cofibrant) invertible $R$-module from each equivalence class.  The object set is a $G$-set with $G$ acting trivially.  That is, each object is itself a $G$-object.   Its morphism $G$-spaces are given by the nonequivariant isomorphisms, with $G$ acting by conjugation.  Under smash product, this gives a symmetric monoidal $G$-category.   We can rectify it to an equivalent (classical) permutative $G$-category and thus a $\sP$-$G$-category.   Its fixed point category 
$\sP ic(R)$ has the same objects, but now the morphisms are given by the nonequivariant spaces of $G$-isomorphisms between them.  Write $R|_{H}$ for $R$ regarded as an $H$-spectrum. As is easily verified, we then have a functor 
$$ \ul{\sP ic} \colon \sO_G^{op} \rtarr \sC at$$
with 
$$ \ul{\sP ic}(R)(G/H) = \sP ic(R|_{H}).$$

Applying the classifying $G$-space functor to  $\sP ic_G(R)$, we obtain the $\sP^{top}$-$G$-space 
${\bf{Pic_G}}(R)$ and its fixed point orbital presheaf  ${\ul{\bf{Pic}}_G}(R) = \bR{\bf{Pic}_G}(R)$, which is a $\bP^{pre}$-algebra by \autoref{Acom}.  

Remember that $G$ acts trivially on objects, hence so does $H\subset G$, whereas the $H$- fixed points of morphism spaces are obtained by restricting to those morphisms which are $H$-maps.     Inspection of the behavior of $B$ on fixed points shows that
$${\ul{\bf{Pic}}_G}(R)(G/H) = {\bf{Pic}}(R|_{H})$$ 
because
$$  (\sP ic_G(R))^H=\sP ic(R|_{H}).$$
This implies the following result.

\begin{prop}\label{picprop}   Applying the classical infinite loop space machine to the $\sP^{top}$-$G$-space ${\bf{Pic_G}}(R)$ gives a classical $G$-spectrum 
$$ pic(R) = \bK \sP ic_G(R)$$ 
such that
$$  (\bK \sP ic_G(R))^H \htp \bK (\sP ic(R|_{H})) $$
for all $H\subset G$.  
\end{prop}

We want genuine $G$-spectra, so we consider
$$\sC at(\sE G,\sP ic_G(R)) = \sP ic_G(R)_G.$$
We denote its classifying space by ${\bf{Pic_G}}(R)_G$, which is a  $\sP_G^{top}$-$G$-space.  We could apply the direct infinite loop space machine  of \autoref{machine}, but the resulting $G$-spectrum would not have the desired fixed point spectra.  We instead apply the functor $\bR$ to ${\bf{Pic}}_G(R)_G$.  That gives a 
$\bP_G^{pre}$-algebra orbital presheaf by \autoref{Acom}.  

\begin{defn}\label{PicG}  Applying the machine of \autoref{thm:pre-machine} to  $\bR{\bf{Pic}}_G(R)_G$ gives the  genuine Picard $G$-spectrum $pic_G(R)$.   
\end{defn}

This raises the following analog of \autoref{Spec90}.

 \begin{quest}\label{Pic90} Let $R$ be a commutative ring $G$-spectrum.   What conditions ensure that 
 $\sP ic_G(R)$ is amenable, so that the natural map
 $$ pic(R)^H = (\bK \sP ic_G(R))^H  \rtarr  \io^*\big(\bK_G\sP ic_G(R)_G\big)^H = pic_G(R) ^H$$
 is a weak equivalence for all $H\subset G$?  By \autoref{picprop}, the domain is weak equivalent to  
 $\bK (\sP ic(R|_{H}))$.
 \end{quest}
 
 \begin{rem}\label{Comp1}  This is a question about homotopy groups $\pi_n^H$.  Checking that the diagram
 $$\xymatrix{
 \SI\bK(\sG\sL(R)) \ar[d] \ar[r]^-{\io}  & \SI\io^*\bK_G(\sG\sL(R)_G) \ar[d] \\
 \bK(\sP ic_G(R)) \ar[r]^-{\io}   & \io^* \bK_G (\sP ic_G(R)_G)\\}
 $$
 commutes, where the vertical arrows are connective covers, the question in positive degrees is equivalent to \autoref{Spec90}. In degree $0$,  the bottom domain becomes $\ul{Pic}(R)$ and details of the  target are given in \cite{Mer, Thom}.  We have not yet taken a closer look. 
 \end{rem}

\subsection{Azumaya algebras and nonequivariant Brauer spectra}\label{BrR}  Here we describe nonequivariant Brauer spectra.  Up to notation, we largely follow the constructions of Baker, Richter, and Szymik \cite{BRS,  Szymik}, which in part follow Schwede and Shipley \cite{SchMor, ShipMor}. However, as in \cite{GM0} we use topological rather than simplicial enrichment.    The cited papers generalize classical algebraic theory to structured ring spectra. Again let $R$ be a (cofibrant) commutative $S$-algebra \cite{EKMM}.  A  key definition is that of a topological version of an Azumaya $R$-algebra.  Write $A^{o}$ for the opposite of an $R$-algebra $A$ and $F_R(A,A)$ for the endomorphism $R$-algebra of $A$. (It is an $E_1$ $R$-algebra).  The left action of $A$ on $A$ and the right action of $A^{o}$  induce a map  $\mu\colon A\sma_R A^{o} \rtarr F_R(A,A)$ of $R$-algebras.   Say that an $R$-module $X$ is faithful if for an $R$-module $Y$, $X\sma_R  Y =  \ast$ implies $Y=\ast$.  

\begin{defn}\label{Azu}  An $R$-algebra $A$ is a (topological) Azumaya $R$-algebra if the following conditions hold.
\begin{enumerate} [(i)]
\item  $A$ is a dualizable $R$-module.
\item $\mu\colon A\sma_R A^{o} \rtarr F_R(A,A)$ is a weak equivalence.
\item  $A$ is a faithful $R$-module.
\end{enumerate}
\end{defn} 

The idea of (ii) is that $F_R(A,A)$ should be equivalent to  $A\sma_R F_R(A,R)$ and $F_R(A,R)$ should be  equivalent to $A^{o}$.
We define a category $\sB r(R)$.  Its objects are the Azumaya  $R$-algebras $A$, thought of as  the categories  $\sM_A$ of $A$-modules.   It is enriched over $\sM_R$,  with Hom objects the $R$-modules of  $R$-equivalences $\sM_A \rtarr \sM_B$.  The $R$-equivalences send an $A$-module $M$ to the $B$-module  $M\sma_R E$ for some invertible $(A,B)$-bimodule $E$ whose $R$-module structures, induced by its $A$ and $B$-module structures agree.  Note that $\sB r(R)$ is symmetric monoidal under  $\sma_R$.   

The Brauer group is then defined via a topological version of Morita equivalence.  For an $R$-module $M$, let $\sE_R(M) = F_R(M,M)$.   By \cite[Proposition 2.2]{BRS}, if $M$ is a dualizable and faithful cofibrant $R$-module, then  (a cofibrant replacement of) $\sE_R(M)$ is an Azumaya $R$-algebra.   That gives one starting point for the $\infty$ categorical reinterpretation in \cite{GL}.

\begin{defn}\label{Mor} Say that Azumaya $R$-algebras $A$ and $B$ are Morita equivalent if there are dualizable and faithful cofibrant $R$-modules $M$ and $N$ such that
the $R$-algebras $A\sma_R \sE_R(M)$ and  $B\sma_R \sE_R(N)$ are equivalent.  The set (and it is a set) of equivalence classes is the Brauer group $Br(R)$. Its product 
is induced by $\sma_R$, and it is abelian.
\end{defn}

Define the Brauer space $\mathbf{Br}(R)$ to be the classifying space of the category $\sB r(R)$.  Its set of components is an abelian group.   By \cite[Proposition 5.1]{Szymik},  the space of self-equivalences of the category $\sM_R$ is equivalent to the space of invertible $R$-modules.  Observe that invertible $R$-modules are dualizable and faithful.  By \cite[Proposition 5.2]{Szymik}, 
$\pi_0({\mathbf{Br}}(R))$ is isomorphic to $Br(R)$.  

Since  $\sB r(R)$ is symmetric monoidal, we obtain a spectrum ${\mathbf{br}}(R)$ whose zeroth space is equivalent to ${\mathbf{Br}}(R)$.  The component of the base object $R$ in ${\mathbf{Br}}(R)$  is equivalent to the classifying space of the automorphism group  $\sE(R)$, which is equivalent to $\sP ic(R)$. Therefore we can view  
$\bf{br}(R)$ as a delooping of ${\mathbf{Pic}}(R)$, and  $\pi_{n+1}{\bf{Br}}(R) \iso \pi_n {\mathbf{Pic}}(R)$ if $n\geq 0$.\footnote{We have corrected some typos at  the end of \cite[Section 5]{Szymik}.} 
 Thus $\SI {\mathbf{pic}}(R)$ is a connective cover of ${\mathbf{br}}(R)$.        
 
\subsection{Equivariant Brauer spectra}\label{GBrR}

We would like an analogous picture equivariantly, and we mimic the equivariant Picard spectra story, again starting with a commutative $S_G$-algebra $R$.   The definition of Azumaya algebras carries over equivariantly, working in the category of $R$-modules and algebras in the category of $G$-spectra.  The definition of Morita equivalence also carries over, giving us the Brauer group $Br(R)$.

Form the topological $G$-category $\sB r_G(R)$ of Azumaya $R$-algebras. Its objects are one chosen (cofibrant) Azumaya $R$-algebra from each equivalence class.  The object set is a $G$-set with $G$ acting trivially.  That is, each object is itself a $G$-object.   Its morphism $G$-spaces are given by the nonequivariant isomorphisms, with $G$ acting by conjugation.  Checking that the functor $\sma_R$ preserves Azumaya algebras by checking the three required properties, we see that $\sB r_G(R)$ is a classical symmetric monoidal $G$-category.   We can rectify it to an equivalent classical permutative $G$-category and thus a $\sP$-$G$-category.   Its fixed point category  $\sB r(R)$ has the same objects, but now the morphisms are given by the nonequivariant spaces of $G$-isomorphisms between them.  Just as for $\sP ic$,  we then have a functor 
$$ \ul{\sB r} \colon \sO_G^{op} \rtarr \sC at$$
with 
$$ \ul{\sB r}(R)(G/H) = \sB r(R|_{H}).$$

Applying the classifying $G$-space functor to  $\sB r_G(R)$, we obtain the $\sP^{top}$-$G$-space 
${\bf{Br_G}}(R)$ and its fixed point orbital presheaf  ${\ul{\bf{Br}}_G}(R) = \bR{\bf{Br}_G}(R)$, which is a $\bP^{pre}$-algebra by \autoref{Acom}.  

Again as for $\sP ic$, 
$${\ul{\bf{Br}}_G}(R)(G/H) = {\bf{Br}(R|_{H}})$$ 
because
$$  (\sB r_G(R))^H=\sB r(R|_{H}).$$

This implies the following result.

\begin{prop}\label{brprop}   Applying the classical infinite loop space machine to ${\bf{Br_G}}(R)$ gives a classical $G$-spectrum 
$$ br(R) = \bK \sB r_G(R)$$ 
such that
$$  (\bK \sB r_G(R))^H \htp \bK (\sB r(R|_{H})) $$
for all $H\subset G$.  
\end{prop} 

We want genuine $G$-spectra, so we consider
$$\sC at(\sE G,\sB r_G(R)) = \sB r_G(R)_G.$$
We denote its classifying space by ${\bf{Br_G}}(R)_G$, which is a  $\sP_G^{top}$-$G$-space.  We could apply the direct infinite loop space machine  of \autoref{machine}, but the resulting $G$-spectrum would not have the desired fixed point spectra.  We instead apply the functor $\bR$ to ${\bf{Br}}_G(R)_G$.  That gives a $\bP_G^{pre}$-algebra orbital presheaf by \autoref{Acom}.     

\begin{defn}\label{BrG}  Applying the machine of \autoref{thm:pre-machine} to  $\bR{\bf{Br}}_G(R)_G$ gives the  genuine Brauer $G$-spectrum $br_G(R)$.   
\end{defn}

This raises the following analog of Questions \ref{Spec90} and \ref{Pic90}. 

 \begin{quest}\label{Br90} Let $R$ be a commutative ring $G$-spectrum.   What conditions ensure that $\sB r_G(R)$ is amenable, so that the natural map
 $$ br_G(R)^H =(\bK\sB r_G(R))^H  \rtarr  \io^*(\bK_G\sB r_G(R))^H =br_G(R)^H$$
 is a weak equivalence for all $H\subset G$?
 \end{quest}
 
 \begin{rem}\label{Comp2}  This is a question about homotopy groups $\pi_n^H$.  Checking that the diagram
 $$\xymatrix{
 \SI\bK(\sP ic_G(R)) \ar[d] \ar[r]^-{\io}  & \SI\bK_G(\sP ic_G(R)_G) \ar[d] \\
\bK(\sB r_G(R)) \ar[r]^-{\io}   & \io^*\bK_G(\sB r_G(R)_G)  \\}
 $$
 commutes, where the vertical arrows are connective covers, the question in positive degrees is equivalent to \autoref{Pic90}. In degree $0$,  the bottom domain becomes $\ul{B r}(R)$ and a direct check of the  target has not yet been pursued. \end{rem}

Observe that the underlying classical $G$-spectrum $\SI^2gl_1(R)$ is the $1$-connected cover of the classical $G$-spectrum $pic(R)$.  

\section{The monadic Elmendorf construction and the presheaf machine}
\label{sec:presheaf-machine}
Recalling Notations \ref{GrothGcats} and \ref{LAG}, we work mainly in the context of based spaces and based $G$-spaces in this section.   As before, let $G\sO$ be the orbit category of (unbased) $G$-sets $G/H$ and $G$-maps.   Let $\bO_+: G\sO \rtarr G \sT$ be the (covariant)  functor that sends the orbit $G/H$, viewed as an object of $G\sO$, to the based $G$-space  $G/H_+$.  The original (based) Elmendorf functor  is the {\em categorical} bar construction that sends an orbital presheaf $\sX\colon \sO^{op} \rtarr \sT$ to the based $G$-space
$$  \originalE(\sX) = B(\sX, G\sO, \bO_+).$$
Its $G$-action is induced by the $G$-action on the $G$-spaces $G/H_+$.
In \autoref{sec:an-altern-elmend}, following Rubin \cite{RubElm}, we enlarge $\originalE$ to a new functor $\ourE$ which works better
with operad actions. We emphasize that in this section we use the categorical and not the monadic 
bar construction used earlier in the paper.  Thus the central term is a category rather than a 
monad, and the left and right terms are contravariant and covariant functors defined on that category.
See for example \cite[Section 3.1]{MMO}.\footnote{That short section is readable independently of the rest of the paper.} 

Recall that we have monads $\Cmon $ on $G\sT$ and $\Cp$ on
$G\sO^{op}[\sT]$.
In \autoref{sec:comp-natur-transf}, we construct a comparison natural transformation
\begin{equation}\label{maplambda}
\la \colon \ourE\Cp \rtarr \bC \ourE
\end{equation}
that is a weak equivalence when applied to any $\sX$.
In  \autoref{sec:final}, we use $\lambda$ to define a functor 
\begin{equation*}
  \ourE_{\bC}: \Cp\big[\OG^{op}[\sT]\big] \rtarr \Cmon[G\sT]
\end{equation*}
that suitably extends $\la$ and leads to the proof of  \autoref{thm:pre-machine}.

\subsection{An enlarged Elmendorf functor $\ourE$}\label{sec:an-altern-elmend}

There are several possible ways to define $\ourE$. The essential point is to generalize from orbits to general finite $G$-sets in such a way that the pullbacks that feature in the relevant Grothendieck constructions can be used to construct the map $\la$ of  \autoref{maplambda}.   Write $G\UP_*$ for $G\LA_*^{op}$, where $G\LA_{\star}$ is the Grothendieck category constructed in \autoref{LambdaStar}  by taking $\sO$ there to be the category $\OG$ here and taking $\sG_{G/H}$ there to be $G\LA_{G/H}$.  
Thus the objects of $G\UP_*$ are the pairs $(G/H, S)$, where $S$ is a finite $G$-set over $G/H$, given by a $G$-map $\pi\colon S \rtarr G/H$.  A morphism $(\ph, \ps)\colon (G/H, S) \rtarr (G/K, T)$ consists of a $G$-map $\ph\colon G/H\rtarr G/K$ and a map $\ps\colon \ph^* T \rtarr S$ of $G$-spaces over $G/H$.  Composition is as defined in \autoref{LambdaStar}.

The $G$-maps $\pi$ together define a functor $G\UP_* \rtarr \OG$.   It is a left adjoint with right adjoint $\io$ defined by $\io(G/H, S) = (G/H, G/H)$, with $\pi = \id_{G/H}$. That is, $\io$ sends $(G/H, S)$ to $G/H$, regarded as a $G$-set over itself.  Remember that the classifying spaces of adjoint categories are canonically homotopy equivalent. 

We define
$$  \ourE(\sX) = B(\sX\com \pi, G\UP, \bO_+\com \pi),$$
where the action of $G$ on orbits in the third variable gives the action of $G$.   The bar construction $\sE(\sX)$ is the realization of a simplicial $G$-space such that composition of faces, application of $\sX$ on orbits, and passage to $H$-fixed points for $H\subset G$ gives a natural map of orbital presheaves
$$ \nu\colon \bR \ourE(\sX) \rtarr \sX. $$
The functor $\pi \colon G\LA_{\star}^{op}\rtarr \OG$ induces a natural map
$$\pi\colon \ourE \rtarr \originalE$$
such that $\nu_{orig}\com \pi = \nu$, where $\nu_{orig}\colon \bR\originalE(\sX) \rtarr \sX$.  

In general, the input triple of the categorical bar construction transforms by a category of elements construction to a single category whose classifying space is the original two-sided bar construction \cite[Section 3.1]{MMO}.  For example, the category of elements of $\originalE(\sX)$ has objects
$$\coprod_{H} \sX(G/H) \times G/H_+$$
and morphisms
$$ \coprod_{(H,K)} \sX(G/K) \times \OG(G/H,G/K)\times G/H$$
and similarly for $\ourE(\sX)$.  It is easily verified that the adjunction $(\pi,\io)$ induces an adjunction between these categories of elements.  That implies that $\pi$ is a $G$-homotopy equivalence.  Therefore $\nu$ is a homotopy equivalence since $\nu_{orig}$ is a homotopy equivalence.  Alternatively, one can directly imitate the details of the proof in
\cite[Theorem V.3.2]{EHCT} that $\nu_{orig}$ is a homotopy equivalence.

\subsection{The comparison natural transformation $\lambda$}
\label{sec:comp-natur-transf}
We prove the following result in this subsection.

\begin{thm}\label{lambda} For orbital presheaves $\sX$, there is a natural map of $G$-spaces
$$\la\colon \ourE (\Cp \sX) \rtarr \bC\ourE(\sX)$$
such that the following diagram commutes, hence $\la$ is a weak equivalence. 
\begin{equation}\label{ladiag}
  \begin{tikzcd}
    \bR \ourE\Cp \sX \ar[r," \bR \lambda"] \ar[d, "\nu"']& \bR \Cmon 
    \ourE \sX \ar[r,"\cong"] & \Cp \bR \ourE \sX \ar[d,
    "\Cp \nu"] \\
    \Cp \sX \ar[rr, equal]& & \Cp \sX
  \end{tikzcd}
\end{equation}
\end{thm}
\begin{proof} Write $(\ul{\ph}, \ul{\ps})$ for a sequence 
$$\ps_r\colon T_r\rtarr T_{r-1}  \, \, \text{over} \, \, \ph_r\colon G/H_r \rtarr G/H_{r-1}, \, \, 1\leq r \leq q,$$ 
of morphisms of $G\UP_*$.  Thus we have a sequence of pullback squares
$$ \xymatrix{
T_q \ar[r]^-{\ps_q} \ar[d]_{\pi} & T_{q-1} \ar[r] \ar[d]_{\pi}& \cdots  \ar[r] & T_1 \ar[d]_{\pi} \ar[r]^-{\ps_1} & T_0 \ar[d]^{\pi} \\
G/H_q \ar[r]_-{\ph_q} & G/H_{q-1}  \ar[r] & \cdots \ar[r] & G/H_1 \ar[r]_-{\ph_1} & G/H_0 \\}
$$
The domain and target of $\la$ are realizations of simplicial $G$-spaces 
with $q$-simplices of the form
$$  B_q\big((\Cp \sX)\com \pi, G\LA_{\star}, \bO_+\com \pi\big) = \coprod_{(\ul{\ph}, \ul{\ps})} 
\big(\Cp\sX(G/H_0)\times (\ul{\ph}, \ul{\ps})\times (G/H_q)_+\big)/(\sim)$$
and, using that $\bC$ commutes with geometric realization, 
$$\bC B_q(\sX\com \pi, G\LA_{\star}, \bO_+\com \pi) = \coprod_{(\ul{\ph}, \ul{\ps})} 
\big(\sX(G/H_0)\times (\ul{\ph}, \ul{\ps})\times (G/H_q)_+\big)/(\sim)$$
Recall that
$$\Cp_{G/H}\sX   =\big(\coprod_{T \in G\LA_{G/H}} \sCpGH(T) \times_{\SI_{\bn^{\al}}} (\bS_{G/H}\sX)(T)\big) /(\sim).$$
Write $y\in (\bP\Cp\sX)(T_0)$ as the image  $[c,y]$ of a pair $(c,y)$, where 
$$c\in \sCpre_{G/H_0}(T_0) \ \ \text{and} \ \  y\in (\bS_{G/H_0}\sX)(T_0).$$
By \autoref{Cpre1}, we can interpret $c$ to be an element of $\sC(n)$ fixed by $\GA_{\al}$.  Regarding $\bS_{\star}\sX$ as a functor 
$G\LA_{\star} \rtarr \sT$, we have $x = \ul{\psi}^*(y)$ in $\bS_{G/H_q}(T_q)$.   On $q$-simplices, for $t\in (G/H_q)_+$, we define
$$\la([c,y], (\ul{\ph}, \ul{\ps}), t) = [c, \big(x,  (\ul{\ph}, \ul{\ps}), t\big)] $$
where we again use $[-,-]$ to denote passage to equivalence classes.  It is an exercise to check that $\la$ is a well-defined map of $G$-spaces.

Commuting passage to fixed points past realization,  the commutativity of \autoref{ladiag} is checked by inspection. Its unlabelled isomorphism is given by \autoref{prop:phi-C-commute}. The vertical maps in \autoref{ladiag} are weak equivalences, the right one because $\Cp$ preserves weak equivalences.  It follows that $\bR \lambda$ is a levelwise weak equivalence, which means that $\lambda$ is a weak equivalence.
\end{proof}

\subsection{The presheaf machine}\label{sec:final}

To go from here to monads, we need the following result.

\begin{prop}
  The following diagrams are commutative: 
\begin{equation*}
\xymatrix{
\ourE \ar[r]^-{\ourE\et} \ar[dr]_{\et} & \ourE\Cp \ar[d]^{\la} \\
& \bC\ourE }
\, \, \text{and} \, \, 
\xymatrix{
    \ourE\Cp \Cp \ar[r]^-{\la} \ar[d]_{\bE\mu} &
    \bC \ourE\Cp \ar[r]^-{\bC \la} &
    \bC \bC \ourE \ar[d]^{\mu} \\
    \ourE\Cp \ar[rr]_-{\la} & & \bC \ourE.}
\end{equation*}
\end{prop}
\begin{proof}
  The first is obvious and the second is a tedious verification that we will omit. It amounts to  unraveling the definition
  of $\Cp$ as in \autoref{Acom} and checking the compatibility with $\lambda$.
\end{proof}

Recall that a $\Cp$-functor $F$ is a functor with a right action $F\Cp \rtarr F$ such that the evident diagrams commute.
It is an elementary diagram chase to deduce the following result.
\begin{cor} The functor $\bC\ourE$ is a $\Cp$-functor with action the composite
$$ \xymatrix@1{ \bC\ourE\Cp     \ar[r]^-{\bC \la}  &  \bC\bC \ourE \ar[r]^{\mu} & \bC\ourE. \\}
$$
Moreover, $\la\colon \ourE \Cp  \rtarr \bC\ourE$ is a map of $\Cp$-functors.
\end{cor}
This allows the following definition.
\begin{defn} 
Define a functor $\ourEonAlg\colon \Cp\big[G\sO^{op}[\sT]\big] \rtarr  \bC[G\sT]$ by
$$ \ourEonAlg(\sX) = B(\bC\ourE, \Cp, \sX).$$
The target is the realization of a simplicial $\bC$-space and is therefore a $\bC$-space \cite[Theorem 12.2]{MayGeo}.   The natural maps
$$ \xymatrix@1{ \ourE \sX & \ar[l] B(\ourE\Cp, \Cp, \sX) \ar[rr]^-{B(\la,\id,\id)} & & B(\bC\ourE, \Cp, \sX) = \ourEonAlg(\sX)\\}
$$
are equivalences of $G$-spaces, the first by a standard extra degeneracy argument  \cite[Theorem 9.10(i)] {MayGeo} and the second because realization of a levelwise equivalence of Reedy cofibrant simplicial $G$-spaces is an equivalence. After application of $\bR$, that implies a (weak) equivalence between the $\Cp$-algebras $\sX$ and $\bR\ourEonAlg(\sX)$. \end{defn}

With this definition in place, the derivation of \autoref{thm:pre-machine} from \autoref{machine} is immediate.

\section{Appendix: The categorical framework}\label{CAT}

\subsection{The Grothendieck construction and fiberwise tensor products}\label{Ofunctors}

We assume that we have a Grothendieck context:

\begin{cont}\label{catassA} Let $\sO$ be a category regarded as a $2$-category with no non-identity $2$-cells and let $\mathrm{Cat}$ be the $2$-category of categories. We assume given a contravariant pseudofunctor $\sG_{\bullet}\colon \sO \rtarr \mathrm{Cat}$.  For a morphism $\ph\colon Y\rtarr X$  in $\sO$, an object $T\in \sG_X$, and a morphism $f\colon T \rtarr U$ in $\sG_X$, we write 
$$ \sG_{\ph}(T) = \ph^*(T)\in \sG_Y \ \ \text{and} \ \ \sG_{\ph}(f) = \ph^*(f) \colon \ph^*(T) \rtarr \ph^*(U).$$ 
\end{cont}

Let $\sV$ be a category. Assuming we are in \autoref{catassA}, the following two definitions give the kind of functors we have concentrated on.
\begin{defn}\label{defn:fiberOversG}
A covariant Grothendieck $\sO$-functor $S_{\star}$ consists of \begin{enumerate}[(i)]
 \item a functor $S_{X}\colon \sG_{X} \rtarr \sV$ for each object $X$ of $\sO$ and
 \item a natural transformation 
                   $	S_\phi\colon S_{Y} \rtarr S_{X}\circ\phi^{*}$ for each morphism \linebreak
                   $\phi\colon Y \rtarr X$ in $\sO$ such that 
                   $S_{id} = \id$ and, for $\rh\colon Z\rtarr Y$, $$S_{\ph\com \rh} = S_{\rh}\com S_{\ph}.$$ 
 More explicitly, the latter composite is
\begin{equation*}
 \begin{tikzcd}
   S_{X} \ar[r,"S_{\ph}"] & S_{Y}\circ\phi^{*} \ar[r,"S_{\rho} \circ \phi^{*}"]
  & S_{Z}  \circ \rh^{*} \circ \phi^{*} \iso S_{Z}  \circ (\phi \circ \rh)^{*},
 \end{tikzcd}
\end{equation*}
where the isomorphism is given by the pseudofunctoriality; the associativity of composition is implied by the equality of $2$-cells that is given as part of the definition of a pseudofunctor.
\end{enumerate} 
A morphism $\et\colon S\rtarr S'$ between such functors consists of natural transformations 
$\et_{X}\colon S_{X} \rtarr S'_{X}$ such that the following diagrams commute.
\begin{equation*}
\xymatrix{
S_{X} \ar[r]^-{S_{\ph}} \ar[d]_{\et_{X}} & S_{Y}\com \ph^* \ar[d]^{\et_{Y}}\\
S'_{X} \ar[r]_-{S'_{\ph}} & S'_{Y}\com \ph^* \\} 
\end{equation*} 
Let $\sG_{\star}[\sV]$ denote the category of such functors and morphisms. 
\end{defn}

\begin{defn}\label{defn:cofiberOversG} A contravariant Grothendieck $\sO$-functor $C_{\star}$ consists of
\begin{enumerate}[(i)]
     \item a functor $C_{X}\colon  \sG_{X}^{op} \rtarr \sV$ for each object $X$ of $\sO$ and
     \item a natural transformation $C_\phi\colon C_{X} \rtarr C_{Y}\circ (\phi^{*})^{op}$ for each morphism \linebreak $\phi\colon Y \rtarr X$ in $\sO$ such that $C_{id} = \id$ and, for $\rh\colon Z\rtarr Y$, 
$$C_{\ph\com\rh} = C_{\rh}\com C_{\ph}.$$
More explicitly, the latter composite is
\begin{equation*}
 \begin{tikzcd}
  C_{X} \ar[r,"C_{\ph}"] & C_{Y}\circ (\phi^{*})^{op}\ \ar[r,"C_{\nu} \circ (\phi^{*})^{op}"]
  & C_{Z}  \circ (\rh^{*})^{op} \circ (\phi^{*})^{op} \iso C_{Z}  \circ ((\phi \circ \rh)^{*})^{op}.
\end{tikzcd}
\end{equation*}
\end{enumerate} 
The isomorphism is given by the pseudofunctoriality.
A morphism $\et\colon C\rtarr C'$ between such functors consists of natural transformations 
$\et_{X}\colon C_{X} \rtarr C'_{X}$ such that the following diagrams commute.
\begin{equation*}
\xymatrix{
C_{X} \ar[r]^-{C_{\ph}} \ar[d]_{\et_{X}} & C_{Y}\com (\ph^*)^{op} \ar[d]^{\et_{Y}}\\
C'_{X} \ar[r]_-{C'_{\ph}} & C'_{Y}\com (\ph^*)^{op} \\} 
\end{equation*} 
Let $\sG^{vop}_{\star}[\sV]$ denote the category of such functors and morphisms. 
 \end{defn}
  
\begin{rem}\label{defn:fiberOversGbis} There are Grothendieck categorical fibrations $\sG_{\star}\rtarr \sO^{op}$ and 
$\sG^{vop}_{\star}\rtarr \sO^{op}$, defined in Definitions \ref{LambdaStar} and \ref{LambdaStarv}, such that the categories $\sG_{\star}[\sV]$ and $\sG^{vop}_{\star}[\sV]$ of Definitions \ref{defn:fiberOversG} and \ref{defn:cofiberOversG} are the respective functor categories.  We are interested in these functor categories, but for the most part we have little interest in these Grothendieck categories.  We have adapted the exposition accordingly.  
\end{rem}

\begin{rem} When constructing $\sG_{\star}^{vop}$, we replace the $\sG_X$ by their opposite categories.  Thinking of morphisms of $\sO$ as ``horizontal'' and morphisms of the $\sG_X$ as ``vertical'', we are taking vertical opposite categories, hence the notation.
\end{rem}

 \begin{defn}
   \label{const:tensorOverLA} 
Given contravariant and covariant $\sO$-functors $C_{\star}$ and $S_{\star}$, we define the fiberwise tensor product $\sO$-presheaf $C_{\star}\otimes_{\sG_{\star}}S_{\star}$ by 
$$(C_{\star}\otimes_{\sG_*}S_{\star})(X) = C_{X} \otimes_{\sG_{X}} S_{X}.$$ 
This is contravariantly functorial on $\sO$ by parts (ii) of \autoref{defn:fiberOversG}: for a morphism 
$\phi\colon Y \rtarr X$, we define $\ph^* = (C\otimes_{\sG_*}S)(\ph)$ to be the composite
 \begin{equation*}
   \xymatrix{
     C_{X} \otimes_{\sG_{X}}S_{X}\ar[r]^-{C_{\ph}\otimes S_{\phi}} &
     C_{Y}\com (\ph^{*})^{op} \otimes_{\sG_{X}} S_{Y} \com \phi^{*} \ar[r]
     & C_{Y} \otimes_{\sG_{Y}} S_{Y}.}
    \end{equation*}
 Here the second arrow is given by an application of the universal
 property of the coequalizer that defines the functor $\otimes_{\sG_{X}}$.
 \end{defn}
 
 The following remark is crucial to conceptual understanding.  
 
\begin{rem}\label{emph} We emphasize that on objects $X$ the definition of the fiberwise tensor product depends only on the functors $C_{X}$ and $S_{X}$, so only on parts (i) of Definitions \ref{defn:fiberOversG} and \ref{defn:cofiberOversG}. Despite the notation, we {\em never} take actual tensor products of functors defined using the horizontal as well as the vertical morphisms in our Grothendieck categories. That would destroy the functoriality on $\sO$, which is the key point of the construction.  Parts (ii) of Definitions \ref{defn:fiberOversG} and \ref{defn:cofiberOversG} provide that functoriality and are therefore separated out in those definitions. 
\end{rem}

The following result is a formal property of Grothendieck constructions that can also be checked directly from the explicit constructions of $\sG_{\star}$ and $\sG_{\star}^{vop}$ that we give below.  We emphasize that pseudofunctors are used to define composition in these categories and to prove its associativity. That is their only use in the present theory, which explains why we have never had to consider them before this Appendix.   

\begin{prop} \label{prop:fiberOversG2}
The functors $S_{\star}\colon \sG_{\star} \rtarr \mathscr{sV}$ are the covariant $\sO$-functors $S_{\star}$ and their natural transformations are the morphisms of the category $\sG_{\star}[\sV]$. 
The functors $C_{\star}\colon \sG^{op}_{\star} \rtarr \mathscr{\sV}$ are the contravariant $\sO$-functors $C_{\star}$ and their natural transformations are the morphisms of $\sG^{vop}_{\star}[\sV]$.
\end{prop}

Finally, we recall the explicit definitions of the categories $\sG_{\star}$ and $\sG_{\star}^{vop}$, specializing standard $2$-categorical material.  See Section 8 and especially Theorem 8.3.1 of \cite{bor2} for a careful general exposition of the categorical context. Our starting point is the given pseudofunctor 
$\sG_{\bullet}\colon \sO^{op}\rtarr \mathrm{Cat}(\sV)$.  The cited Theorem 8.3.1 has two versions, covariant and contravariant. The covariant Grothendieck construction applied to $\sG_{\bullet}$ gives a categorical fibration 
$\sG_{\star} \rtarr \sO^{op}$.

\begin{defn}\label{LambdaStar} The objects of $\sG_{\star}$ are the disjoint union over $X\in \sO$ of the objects of the categories $\sG_{X}$; equivalently, they are the pairs $(X,T)$ where $T$ is an object of $\sG_{X}$.
A morphism $(\ph,f)\colon (X,T)\rtarr (Y,U)$ consists of a morphism $\ph\colon Y\rtarr X$ in $\sO$
and a morphism  $f\colon \ph^*T \rtarr U$ in $\sG_Y$.  The composite of $(\ph,f)$ and $(\ps, g)\colon (Y,U) \rtarr (Z,V)$
is $(\ph\com \ps, g\ast f)$, where $g\ast f$ is the composite
$$ \xymatrix @1{ (\ph\com \ps)^*(T) \ar[r]^-{\iso} &  \ps^*\ph^*(T) \ar[r]^-{\ps^*(f)} & \ps^*(U) \ar[r]^-{g} & V \\}. $$
Here the isomorphism is given by pseudofunctoriality, and the associativity of composition follows by diagram chasing from the equality of $2$-morphisms required by pseudofunctoriality.  Projection to the first coordinate gives the functor $\sG_{\star} \rtarr \sO^{op}$.
\end{defn}
  
The contravariant Grothendieck construction applied to the pseudofunctor $\sG_{\bullet}$ gives another categorical fibration 
$\sG_{\star}^{vop} \rtarr \sO^{op}$.  
\begin{defn}\label{LambdaStarv} The objects of $\sG^{vop}_{\star}$ are the same as the objects of $\sG_{\star}$ but the morphisms are obtained by reversing the direction of the morphism $f$ in \autoref{LambdaStar}.  Here morphisms are denoted $(\ph,f)\colon (X,T) \rtarr (Y,U)$, but now $f$ is a morphism $U\rtarr \ph^*(T)$ in $\sG_Y$.  The composite of
$(\ph,f)$ and $(\ps, g)\colon (Y,U) \rtarr (Z,V)$
is $(\ph\com \ps, g\ast^{op} f)$, where $g\ast^{op}f$ is the composite
$$ \xymatrix @1{V \ar[r]^-{g} & \ps^*(U) \ar[r]^-{\ps^*(f)} & \ps^*\ph^*(T) \ar[r]^-{\iso} &  (\ph\com \ps)^*(T)\\}, $$
where the isomorphism is again given by pseudofunctoriality.  Again, projection to the first coordinate gives the functor 
$\sG_{\star}^{vop} \rtarr \sO^{op}$.
\end{defn}

\subsection{A general construction of relevant pseudofunctors}\label{Opsfunctors}
\begin{con}\label{phistar}
Assume that $\sO$ and $\sI$ are subcategories (not necessarily full) of some ambient category $\sF$.  When we perform constructions in $\sF$, like pullbacks, we implicitly assume that they exist. We describe a construction of
a contravariant pseudofunctor $\sI_{\bullet} \rtarr \mathrm{Cat}$. For an object $X$ of $\sO$, define $\sI_X$ to be the category of objects of $\sI$ over $X$ in $\sF$.  That is, an object is a morphism $\pi \colon T\rtarr X$ in $\sF$, where $T\in \sI$. We write $\pi$ generically and generally abbreviate notation by writing objects as $T$.  A morphism $f\colon S \rtarr T$ is a morphism in $\sI$ such that 
$\pi\com f = \pi$ in $\sF$.  Let $\ph\colon Y\rtarr X$ be a morphism of $\sO$.  We define a functor 
$\sI_{\ph} = \ph^*\colon \sI_X \rtarr \sI_Y$. On an object $\pi\colon T \rtarr X$ of $\sI_X$, we define $\pi\colon \phi^{*}(T) \rtarr Y$ via the pullback diagram
\begin{equation}\label{reverse}
\xymatrix{
\ph^*(T) \ar[r] \ar[d]_{\pi} & T \ar[d]^{\pi} \\ 
Y \ar[r]_{\ph} &  X. \\}
\end{equation}
in $\sF$.
On a morphism $f\colon S\rtarr T$ of $\sI_X$, we define 
$$\ph^*(f)\colon \ph^*(S) \rtarr \ph^*(T)$$ 
to be the map over $Y$ given by the map of pullback diagrams in $\sF$ induced by $f$:
$$
\xymatrix{
\ph^*(S) \ar[r] \ar[d]_{\pi} & S \ar[d]^{\pi} \\ 
Y \ar[r]_{\ph} & X \\}
\ \ \rtarr \ \
\xymatrix{
\ph^*(T) \ar[r] \ar[d]_{\pi} & T \ar[d]^{\pi} \\ 
Y \ar[r]_{\ph} &  X. \\}
$$
We require that the morphism $\ph^*(f)$ of $\sF$ be in $\sI$. In our applications, $\sI$ will be the subcategory of injective maps in $\sF$ and the requirement will hold.   Functoriality (that is, vertical functoriality) is easily checked.

It remains to verify pseudofunctoriality (that is, horizontal pseudofunctoriality). See for example \cite[Definition 7.5.1]{bor1} for the definition of a pseudofunctor.  For its unit condition, we can identify $\id_{X}^*$ with $\id \colon \sI_{X} \rtarr \sI_{X}$.  For composable morphisms
$\ph\colon Y\rtarr X$ and $\rh\colon Z\rtarr Y$ of $\sO$, we must define a natural isomorphism $\rh^*\com \ph^* \rtarr (\ph\rh)^*$.  For $S\in \sF_X$, the universal property of the three pullbacks in sight gives the three dotted arrows in the following diagram and then shows that $\xi$ and $\xi^{-1}$ are inverse isomorphisms.
\begin{equation*}
\xymatrix{
\rh^*\ph^*S \ar[rr] \ar[ddr]_{\pi} \ar@<-1ex>@{-->}[dr]_(.65){\xi} & & \ph^*S \ar[dd]^(.35){\pi} |\hole \ar[dr] & \\
& (\ph\rh)^*S \ar[rr] \ar[d]^{\pi} \ar@<-1ex>@{-->}[ul]_(.35){\xi^{-1}} \ar@{-->}[ur]& & S \ar[d]^{\pi}\\
& Z \ar[r]_-{\rh}& Y \ar[r]_-{\si} & X \\}
\end{equation*}
The equality of $2$-morphisms required for a composite of three morphisms of $\sO$ is an immediate verification.  
\end{con} 

\bibliographystyle{alpha}
\bibliography{references}

\end{document}